\documentclass[a4paper,11pt]{article}

\usepackage[latin1]{inputenc}
\usepackage[T1]{fontenc}
\usepackage[english]{babel}
\usepackage{url}
\usepackage{amsmath,amssymb,amsthm}
\usepackage{dsfont,bbm}
\usepackage{slashbox,pict2e}
\usepackage{multirow}
\usepackage{pgf,tikz}   
\usepackage{tikz-cd}   
\usepackage{ifthen}\newboolean{color}
\usepackage{placeins}
\usetikzlibrary{positioning}
\usetikzlibrary{calc}
\usetikzlibrary{matrix}

\usepackage{graphicx}
\usepackage{subfigure}
\tikzset{%
	>=latex, 
	inner sep=0pt,%
	outer sep=2pt,%
	mark coordinate/.style={inner sep=0pt,outer sep=0pt,minimum size=3pt,
		fill=black,circle}%
}
\usepackage{listings}
\newcommand{\cT}{{\mathcal{T}}}

\usepackage[shortlabels]{enumitem}
\theoremstyle{plain}

\newtheorem{proposition}{Proposition}[section]
\newtheorem{theorem}{Theorem}[section]
\newtheorem{lemma}{Lemma}[section]
\newtheorem*{lem*}{Lemma}

\newtheorem*{cor*}{Corollary}
\newtheorem*{teo*}{Theorem}
\theoremstyle{definition}
\newtheorem{definition}{Definition}[section]
\newtheorem{ass}{Assumption}

\newtheorem{remark}{Remark}[section]

\newcommand{\minitab}[2][l]{\begin{tabular}{#1}#2\end{tabular}}
 \usepackage[hmarginratio=1:1]{geometry}
\title{BDDC preconditioners for divergence free virtual element discretizations of the Stokes equations}
\author{Tommaso Bevilacqua, Simone Scacchi}

\begin{document}

\title{BDDC preconditioners for divergence free virtual element discretizations of the Stokes equations
}


\author{Tommaso Bevilacqua \footnote{Dipartimento di matematica, Universit\'a degli studi di Milano, Via Saldini 50, 20133 Milano, Italy; e-mail: tommaso.bevilacqua@unimi.it} \and Simone Scacchi \footnote{Dipartimento di matematica, Universit\'a  degli studi di Milano, Via Saldini 50, 20133 Milano, Italy; e-mail: simone.scacchi@unimi.it}}




\maketitle

\begin{abstract}
	The Virtual Element Method (VEM) is a new family of numerical methods for the approximation of partial differential equations, where the geometry of the polytopal mesh elements can be very general. The aim of this article is to extend the balancing domain decomposition by constraints (BDDC) preconditioner to the solution of the saddle-point linear system arising from a VEM discretization of the two-dimensional Stokes equations. Under suitable hypotesis on the choice of the primal unknowns, the preconditioned linear system results symmetric and positive definite, thus the preconditioned conjugate gradient method can be used for its solution. We provide a theoretical convergence analysis estimating the condition number of the preconditioned linear system. Several numerical experiments validate the theoretical estimates, showing the scalability and quasi-optimality of the method proposed. Moreover, the solver exhibits a robust behavior with respect to the shape of the polygonal mesh elements. We also show that a faster convergence could be achieved with an easy to implement coarse space, slightly larger than the minimal one covered by the theory.  \\

\end{abstract}

\section{Introduction}
\label{intro}
The balancing domain decomposition by constraints (BDDC) preconditioner is an iterative substructuring 
method for the solution of partial differential equations (PDEs), that belongs to the class of
nonoverlapping domain decomposition algorithms \cite{smith2004domain,toselli2006domain}. 
BDDC, first introduced in \cite{dohrmann2003} for elliptic problems, represents an evolution of 
the balancing Neumann-Neumann preconditioner \cite{toselli2006domain}.
We also remark that BDDC presents several features in common with the dual-primal finite element tearing and 
interconnecting (FETI-DP) algorithm. In particular, the BDDC and FETI-DP operators share almost 
the same eigenvalues \cite{li2006feti,brenner2007}, thus they exhibit analogous convergence properties. 
Both BDDC and FETI-DP have been successfully developed for finite and spectral element discretizations of several
physical problems governed by PDEs, see e.g. \cite{klawonn2006,zampini2014dual,dohrmann2016bddc,oh2017}.
In particular, regarding the Stokes equations, they have been studied in \cite{li2006bddc,litu2013}. 
In recent years, BDDC and FETI-DP algorithms have been also extended to various innovative discretizations techniques for PDEs,
such as Mortar discretizations \cite{kim2009}, discontinuous Galerkin methods \cite{dryja2007,canuto2014}, isogeometric analysis \cite{hofer2018,widlund2021}, weak Galerkin methods \cite{tu2018} and virtual element methods \cite{bertoluzza2017bddc,bertoluzza2020}.

The Virtual Element Method (VEM), introduced in the pioneering paper \cite{beirao2013basic}, represents a generalization of the finite element method (FEM), that can easily handle general polytopal meshes. The core idea behind VEM is to use approximated discrete bilinear forms, whose computation requires only the integration of polynomials on the element boundary and interior. The resulting discrete solution is conforming and the accuracy guaranteed by such discrete bilinear forms turns to be sufficient to achieve the correct order of convergence. The advantage of these methods is that they can be applied on a wide choice of general polygonal meshes without the need to integrate complex non-polynomial functions on the elements, keeping an high degree of accuracy.

In the VEM literature only a few studies have focused on the construction and analysis of preconditioners for VEM approximations of PDEs; see \cite{antoniettiMasV.2018,calvo.2018,calvo.2019,dassiS.2020b}). BDDC for VEM discretizations of scalar elliptic problems have been first introduced in \cite{bertoluzza2017bddc,bertoluzza2020} and then extended to mixed formulations of scalar elliptic equations in \cite{dassiZS2022}.
To our knowledge, the development of effective non-overlapping domain decomposition preconditioners for VEM discretizations of the Stokes equations is still an open problem. 

The novelty of the present study is to develop a BDDC preconditioner for the divergence free VEM discretization of the two-dimensional Stokes equations introduced in \cite{da2017divergence}. Our algorithm represents an extension to VEM of the BDDC preconditioner proposed in \cite{li2006bddc} for FEM discretizations of the Stokes equations with discontinuous pressure spaces.
We prove a convergence rate estimate of the preconditioned system, independent of the number of subdomains and polylogarithmic 
with respect to the ratio $H/h$, where $H$ denotes the subdomain size and $h$ the mesh size. 
Such an estimate yields the scalability and quasi-optimality of the resulting algorithm.
Several numerical tests confirm the theoretical estimate and show the robustness of the solver with respect to different
polygonal meshes.

The paper is organized as follows: in Section 2 we introduce the continuous problem and its variational formulation; in Section 3  we describe the VEM discretization; in Section 4 we introduce the domain decomposition tecnique and the BDDC preconditioner; in Sections 5 and 6 we describe the theoretical aspects, while in Section 7 we report several numerical results; finally in Section 8 we draw the conclusions.

\section{Continuous problem}
\label{sec:1}
Let $\Omega \subseteq \mathbb{R}^2$, with $\Gamma = \partial \Omega$, and consider the stationary Stokes problem on $\Omega$ with homogeneous Dirichet boundary conditions: 
\begin{equation}\label{ContinuousProblem}
	\begin{cases}
		\text{Find }(\mathbf{u},p)\text{ such that} \\
		-\nu\mathbf\Delta\mathbf{u} - \nabla p = \mathbf{f} \qquad &\text{ in } \mathit{\Omega} \\
		\text{div }\mathbf{u} = 0  \qquad &\text{ in } \mathit{\Omega} \\
		\mathbf{u}=0 \qquad  &\text{ on } \mathit{\Gamma},
	\end{cases}
\end{equation}
where $\mathbf{u}$ and $p$ are the velocity and the pressure fields, respectively. Furthermore $\mathbf{\Delta}$, div and $\mathbf{\nabla}$ denote the vector Laplacian, the divergence and the gradient operators. Finally, $\mathbf{f}$ represents the external force, while $\nu > 0$ is the viscosity.

Let us consider the spaces:
\begin{align}
	\mathbf{V}:=[H^{1}_{0}(\mathit{\Omega})]^2,\qquad Q:=L^{2}_{0}(\mathit{\Omega})=\bigg \{ q\in L^{2}(\mathit{\Omega})\quad s.t. \quad \int_{\mathit{\Omega}} q \text{ d}\Omega=0 \bigg \}
\end{align}
with norms:
\begin{align}
	\Vert\mathbf{v}\Vert_{1}:=\Vert\mathbf{v}\Vert_{[H^{1}(\mathit{\Omega})]^2}, \quad \Vert q\Vert_{Q}:=\Vert q\Vert_{L^{2}(\mathit{\Omega})}.
\end{align}

We assume $\mathbf{f} \in [H^{-1}(\mathit{\Omega})]^2$, and $\nu\in L^{\infty}(\mathit{\Omega})$ uniformly positive in $\mathit{\Omega}$. Let the bilinear forms $a(\cdot,\cdot): \mathbf{V} \times \mathbf{V} \rightarrow \mathbb{R}$ and $b: \mathbf{V} \times Q \rightarrow \mathbb{R}$ be defined as:
\begin{equation}\label{aFormCont}
	a(\mathbf{u},\mathbf{v}) := \int_{\Omega}\nu \mathbf{\nabla u} : \mathbf{\nabla v}\text{ d}\Omega \qquad \text{for all } \mathbf{u},\mathbf{v} \in \mathbf{V}
\end{equation}
\begin{equation}\label{bFormCont}
	b(\mathbf{v},q) := \int_{\Omega} \text{div } \mathbf{v} q \text{ d}\Omega \qquad \text{for all } \mathbf{u} \in \mathbf{V}, q \in Q.
\end{equation}
Then a standard variational formulation of problem \eqref{ContinuousProblem} reads:
\begin{equation}\label{VarForm}
	\begin{cases}
		\text{find } (\mathbf{u},p) \in \mathbf{V} \times Q \text{ such that} \\
		a(\mathbf{u},\mathbf{v})+b(\mathbf{v},p)=(\mathbf{f},\mathbf{v}) & \text{for all } \mathbf{v} \in \mathbf{V}, \\
		b(\mathbf{u},q)=0 & \text{for all } q \in Q,
	\end{cases}
\end{equation}
where
\begin{align*}
	(\mathbf{f},\mathbf{v}):=\int_{\Omega} \mathbf{f} \cdot \mathbf{v}\text{ d}\Omega.
\end{align*}
It is well-known that:
\begin{itemize}
	\item $a(\cdot,\cdot)$ and $b(\cdot,\cdot)$ are continuous, \textit{i.e.} 
	\[
	\begin{array}{llll}
		|a(\mathbf{u},\mathbf{v})| & \leq & \Vert a \Vert \Vert\mathbf{u}\Vert_1 \Vert\mathbf{v}\Vert_1 & \quad \text{for all } \mathbf{u},\mathbf{v} \in \mathbf{V} \vspace{0.2cm},\\
		|b(\mathbf{v},q)| & \leq & \Vert b\Vert  \Vert\mathbf{v}\Vert_1 \Vert q\Vert_Q & \quad \text{for all } \mathbf{v} \in \mathbf{V} \text{and } q \in Q,
	\end{array}
	\]
	where $\Vert a \Vert$ and $\Vert b \Vert$ are the usual norm of the two bilinear forms;
	\item $a(\cdot,\cdot)$ is coercive \textit{i.e.}, there exists a positive constant $\alpha$ such that 
	\[
	|a(\mathbf{v},\mathbf{v})| \geq \alpha \Vert \mathbf{v} \Vert_1^2  \quad \text{for all } \mathbf{v} \in \mathbf{V};
	\]
	\item the bilinear form $b(\cdot,\cdot)$ satisfies the inf-sup condition \cite{boffi2013mixed}, \textit{i.e.} 
	\begin{equation}\label{infsup}
		\exists \beta>0 \text{ such that} \quad \sup_{\mathbf{v}\in \mathbf{V}, \mathbf{v} \neq \mathbf{0}} \frac{|b(\mathbf{v},q)|}{\Vert\mathbf{v}\Vert_1} \geq \beta \Vert q\Vert_Q \quad\text{for all } q \in Q.
	\end{equation}
\end{itemize}
Therefore, problem \eqref{VarForm} has a unique solution $(\mathbf{u},p) \in \mathbf{V} \times Q$ such that 
\begin{align}
	||\mathbf{u}||_1 + ||p||_Q \leq C ||\mathbf{f}||_{H^{-1}(\mathit{\Omega})},
\end{align}
where the constant $C$ depends only on $\mathit{\Omega}$ and $\nu$; see \cite{boffi2013mixed}.

\section{Virtual element discretization}
\label{sec:2}
We present here the discretization of problem \eqref{ContinuousProblem}, based on the virtual element space introduced in \cite{da2017divergence}, that is designed to solve a Stokes-like problem element-wise.
In particular we will use the reduced space presented in section 5 of \cite{da2017divergence}, that, exploiting the divergence free property of the solution, allows to save a lot of degrees of freedom especially when the polynomial degree $k$ is large. We recall here the definition of the local spaces.
Let $\{\cT_h\}_h$ be a sequence of triangulations of $\Omega$ into general polygonal elements $K$ with 
\begin{align*}
	h_K:=\text{diameter}(K), \quad h:=\sup_{K\in \cT_h}h_K.
\end{align*}

We suppose that, for all $h$, each element $K\in \cT_h$ satisfies the following assumptions:
\begin{itemize}
	\item$(\mathbf{A1})$ $K$ is star-shaped with respect to a ball of radius $\geq \gamma h_k$,
	\item $(\mathbf{A2})$ the distance between any two vertices of K is $\geq c h_K$,
	\item $(\mathbf{A3})$ the triangulation $\cT_h$ is quasi-uniform, \textit{i.e.} there exist positive constants $c_0,\,c_1$ such that for any two elements $K$ and $K'$ in $\cT_h$ we have $c_0\leq h_K/h_{K'}\leq c_1$.
\end{itemize}
where $\gamma$ and $c$ are positive constants. 
\begin{remark}
	These hypotheses could be weakened as in \cite{beirao2013basic}, for example assuming that every $K$ is a union of a finite (and uniformly bounded) number of star-shaped domains, each satisfying ($\mathbf{A1}$).
\end{remark}

We also assume that the scalar viscosity field $\nu$ is piecewise constant with respect to the decomposition $\cT_h$, i.e. $\nu$ is constant on each polygon $K \in \cT_h$.

For $k \in \mathbb{N}$, let us define the spaces:
\begin{itemize}
	\item $\mathbb{P}_k(K)$ the set of polynomials on $K$ of degree $\leq k$,
	\item $\mathbb{B}_k(K):=\{v \in C^0(\partial K ) \text{ s.t. } v_{|e} \in \mathbb{P}_k(e) \quad \forall \text{ edge } e \in \partial K \}$,
	\item $\mathit{G}_k(K):=\nabla(\mathbb{P}_{k+1}(K)) \subseteq [\mathbb{P}_{k}(K)]^2$,
	\item $\mathit{G}_k(K)^\perp \subseteq [\mathbb{P}_{k}(K)]^2$ the $L^2$-orthogonal complement to $\mathit{G}_k(K)$.
\end{itemize}  
On each element $K \in \mathit{T}_h$ we define, for $k\geq2$, the following finite dimensional local virtual element spaces:
\begin{equation}\label{NewLoc}
	\begin{split}
		\mathbf{\widehat{V}}_h^K := \bigg\{ \mathbf{v} \in [H^1(K)]^2 \text{ s.t. } \mathbf{v}_{| \partial K} \in \mathbb{B}_{k}(\partial K)]^2,\\
		\begin{cases} -\nu\mathbf\Delta\mathbf{v} - \nabla s \in\mathit{G}_{k-2}(K)^\perp, \\ \text{div } \mathbf{v} \in \mathbb{P}_{0}(K),\end{cases} \text{for some } s \in L^2(K)\bigg\}
	\end{split}
\end{equation}
and 
\begin{align}
	Q_h^K:=\mathbb{P}_{0}(K).
\end{align}

Now it is possible to introduce suitable sets of degrees of freedom for the local approximations fields.\\
Given a function $\mathbf{v} \in \mathbf{\widehat{V}}_h^K$ we take the following linear operators $\mathbf{D_{\widehat{V}}}$, split into three subsets:
\begin{itemize}
	\item $\mathbf{D_{\widehat{V}}1}$: the values of $\mathbf{v}$ at the vertices of the polygon $K$,
	\item $\mathbf{D_{\widehat{V}}2}$: the values of $\mathbf{v}$ at $k-1$ distinct points of every edge $e\in \partial K$ (for the implementation we will take the $k-1$ internal points of the $(k+1)$-Gauss-Lobatto quadrature rule in $e$),
	\item $\mathbf{D_{\widehat{V}}3}$: the moments of the values of $\mathbf{v}$
	\begin{align*}
		\int_{K}\mathbf{v} \cdot \mathbf{g}^\perp_{k-2}\text{ d}K \qquad \text{for all } \mathbf{g}^\perp_{k-2} \in \mathbf{G}_{k-2}(K)^\perp
	\end{align*}
\end{itemize} 
Furthermore, for the local pressure, given $q\in Q^K_h$, we consider the linear operators $\mathbf{D_Q}$:
\begin{itemize}
	\item $\mathbf{D_Q}$: the moment
	\begin{align*}
		\int_{K}q \text{ d}K.
	\end{align*}
\end{itemize}
Since $\mathbf{D_{\widehat{V}}}$ and $\mathbf{D_Q}$ are unisolvent respectively of 
$\mathbf{\widehat{V}}_h^K$ and $Q_h^K$, we can define the global virtual element spaces:
\begin{align}\label{globdiscvel}
	\mathbf{\widehat{V}}_h:=\{\mathbf{v}\in [H^{1}_{0}(\mathit{\Omega})]^2\quad \text{s.t.}\quad \mathbf{v}_{|K} \in \mathbf{\widehat{V}}^K_h\quad \text{for all } K\in \cT_h\}
\end{align} 
and
\begin{align}\label{globdiscpre}
	Q_h:=\{q\in L^{2}_{0}(\mathit{\Omega}) \quad \text{s.t.}\quad q_{|K} \in Q^K_h\quad \text{for all } K\in \cT_h\},
\end{align} 
with obvious associated sets of global degrees of freedom.

\subsection{Discrete problem}
\label{sec:3}
Referring to \cite{da2017divergence}, we can now state the discrete virtual element problem 
\begin{equation}\label{DiscreteProblem}
	\begin{cases}
		\text{find }(\mathbf{u}_h,p_h)\text{ such that} \\
		a_h(\mathbf{u}_h,\mathbf{v}_h)+b(\mathbf{v}_h,p_h)=(\mathbf{f}_h,\mathbf{v}_h)& \quad \text{for all }\mathbf{v}_h \in \mathbf{\widehat{V}}_h\\
		b(\mathbf{u}_h,q_h) = 0&\quad \text{for all }q_h \in Q_h
	\end{cases}
\end{equation}
By construction the discrete bilinear form $a_h(\cdot,\cdot)$ is stable (uniformly) with respect to the $\mathbf{V}$ norm and also obviously the bilinear form $b(\cdot,\cdot)$. Therefore, to prove the existence and uniqueness of the solution of the problem \eqref{DiscreteProblem} is necessary only a suitable inf-sup condition. For our work, we will only need this condition for the subdomains in which $\Omega$ will be divided into. In this way the local subdomains problem, as weel as the global one, will be well posed.
The proof of the following inf-sup condition could be found in \cite{da2017divergence}.
\begin{proposition}
	Given the discrete spaces $\mathbf{\widehat{V}}_h$ and $Q_h$ defined in (\ref{globdiscvel}) and (\ref{globdiscpre}), there exists a positive $\tilde{\beta}$, independent of h, such that:
	\begin{equation}\label{infsupdisc}
		\sup_{\mathbf{v}_h\in \mathbf{\widehat{V}}_h, \mathbf{v}_h \neq \mathbf{0}} \frac{|b(\mathbf{v}_h,q_h)|}{\Vert\mathbf{v}_h\Vert_1} \geq \tilde{\beta} \Vert q_h\Vert_Q \quad\text{for all } q_h \in Q_h.
	\end{equation}
\end{proposition}
A consequence of the previous proposition is the following statement.
\begin{theorem}
	Problem \eqref{DiscreteProblem} has a unique solution $(\mathbf{u}_h,p_h)\in \mathbf{\widehat{V}}_h\times Q_h$, verifying the estimate
	\begin{align}
		\Vert \mathbf{u}_h \Vert_1 + \Vert p_h \Vert_Q \leq C\Vert \mathbf{f}\Vert_0\text{.}
	\end{align}
\end{theorem}

We have also a convergence result

\begin{theorem}
	Let $(\mathbf{u},p)\in \mathbf{V}\times Q$ be the solution of problem \eqref{VarForm} and $(\mathbf{u}_h,p_h)\in \mathbf{\widehat{V}}_h\times Q_h$ be the solution of problem \eqref{DiscreteProblem}. Then it holds
	\begin{align}
		\Vert \mathbf{u}-\mathbf{u}_h \Vert_1 \leq C h^k(|\mathbf{f}|_{k-1}+|\mathbf{u}|_{k+1})
	\end{align}
	and
	\begin{align}
		\Vert p-p_h \Vert_Q \leq C h^k(|\mathbf{f}|_{k-1}+|\mathbf{u}|_{k+1}+|p|_{k})\text{.}
	\end{align}
\end{theorem}

\section{Construction of the BDDC preconditioner}
\label{sec:4}
In this section, we first divide the domain $\Omega$ into subdomains and introduce appropriate function spaces, paragraph 4.1. Then, in paragraph 4.2, we show how the global interface saddle-point problem takes form and then in 4.3 we define the BDDC preconditioner that allows us to use a preconditioned conjugate gradient method (PCG) for its solution.

\subsection{Domain decomposition}
\label{sec:5}
We decompose the domain $\Omega$ into $N$ non-overlapping subdomains $\Omega_i, i=1,2,...N$, of characteristic diameter $H$. Each subdomain is a union of shape regular elements and the nodes on the boundaries of neighboring subdomain match across the interface $\Gamma = (\cup\partial\Omega_i)\setminus \partial\Omega$; we define also $\Gamma_i = \partial\Omega_i \cap \Gamma$ as the interface of an individual subdomain $\Omega_i$.
According to [5], where more details could be found, we recall two requirements on the  subdomain partition:
\begin{itemize}
	\item \textbf{(S1)} Each subdomain $\Omega_i$ is the union of polygonal elements of the triangulation $\cT_h$ and the number of polygons forming an individual subdomain is uniformly bounded;
	\item \textbf{(S2)} If a face of a subdomain intersects $\partial\Omega$, then the measure of this set is comparable to that of $\partial\Omega_i$. Similarly, if an edge of a subdomain intersects $\partial\Omega$, the length of this intersection is bounded from below in terms of the diameter of $\partial\Omega_i$. 
\end{itemize} 

\begin{figure}[!b]
	\centering
	\includegraphics[width=5.5cm]{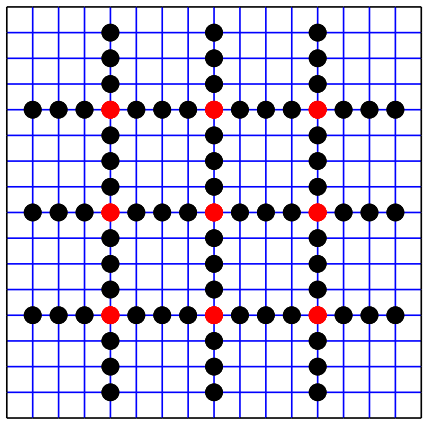}
	\caption{Interface of the subdomains (excluding the nodes on the boundary): red circles indicate the vertices of the subdomains, whereas black circles indicate the remainder interface nodes.}
	\label{splitfig}
\end{figure} 

Restricting to the two-dimensional case, although the theory of iterative substructuring (\cite{toselli2006domain} Section 4.2) does not cover the general cases where the boundary of a subdomain is not a straight line (as we have in our implementation, since we use general polygonal meshes), we can anyway define vertices and interface relatively easily. We say that a node $x$ belongs to the interface of a subdomain if it belongs to at least two subdomains, while a node $x$ is a vertex of a subdomain if it belongs to more than two subdomains (Figure \ref{splitfig}). This is the rule that we used in the implementation to split our mesh in the different subdomains.

\subsection{Decomposition of the virtual element spaces}
\label{sec:6}
The discrete variational problem \eqref{DiscreteProblem} can be written, in matrix form, as the following saddle-point linear system:
\begin{equation}\label{MatrixForm}
	\left[
	\begin{array}{cc}
		A & B^T\\
		B & 0\\
	\end{array}
	\right]
	\left[
	\begin{array}{c}
		\mathbf{u}\\
		p\\
	\end{array}
	\right]
	=
	\left[
	\begin{array}{c}
		\mathbf{f}\\
		0\\
	\end{array}
	\right]
\end{equation}
where the matrices A and B are associated with the discrete bilinear forms $a_h(\cdot,\cdot)$ and $b(\cdot,\cdot)$.
In the remainder of the paper, we omit the underscore $h$ since we will always refer to the finite dimensional space and so we write $\mathbf{\widehat{V}}\times Q$ instead of $\mathbf{\widehat{V}}_h\times Q_h$, only for sake of simplifying the notation.
Referring to the notations of the previous section, we naturally split the degrees of freedom (dofs) of the velocity components into boundary dofs ($\mathbf{D_V1}$ and $\mathbf{D_V2}$) and interior dofs ($\mathbf{D_V3}$ and $\mathbf{D_V4}$).
Following the notations introduced in \cite{li2006bddc}, we decompose the discrete velocity and pressure space $\mathbf{\widehat{V}}$ and $Q$ into:
\begin{equation}\label{discVQ}
	\mathbf{\widehat{V}} = \mathbf{V}_I \bigoplus \mathbf{\widehat{V}}_\Gamma\text{,}
	\quad Q = Q_I\bigoplus Q_0\text{.}
\end{equation}
$\mathbf{V}_I$ and $Q_I$ are direct sums of subdomain interior velocity spaces $\mathbf{V}_I^{(i)}$, and subdomain interior pressure spaces $Q_I^{(i)}$, respectively, i.e.,
\begin{equation}\label{discViQi}
	\mathbf{V}_I = \bigoplus_{i=1}^{N} \mathbf{V}_I^{(i)}\text{,}\quad Q_I = \bigoplus_{i=1}^{N} Q_I^{(i)}\text{.}
\end{equation}
The elements of $\mathbf{V}_I^{(i)}$ have support in the subdomain $\Omega_i$ and vanish on its interface $\Gamma_i$, while the elements of $Q_I^{(i)}$ are restrictions of elements in Q to $\Omega_i$. $\mathbf{\widehat{V}}_\Gamma$ is the space of the traces on $\Gamma$ of functions in $\mathbf{\widehat{V}}$ and $Q_0$ is the subspace of $Q$ with constant values $q_0^{(i)}$ in the subdomain $\Omega_i$.
We denote the space of interface velocity variables of the subdomain $\Omega_i$ by $\mathbf{V}_\Gamma^{(i)}$, and the associated product space by $\mathbf{V}_\Gamma = \prod_{i=1}^{N}\mathbf{V}_\Gamma^{(i)}$; generally functions in $\mathbf{V}_\Gamma$ are discontinuous across the interface. $R_\Gamma^{(i)}:\mathbf{\widehat{V}}_\Gamma\rightarrow\mathbf{V}_\Gamma^{(i)}$ is the operator which maps functions in the continuous interface velocity space $\mathbf{\widehat{V}}_\Gamma$ to their subdomain components in the space $\mathbf{V}_\Gamma^{(i)}$. We denote the direct sum of the $R_\Gamma^{(i)}$ with $R_\Gamma$.\\
With the decomposition of the solution space given in \eqref{discVQ}, the global saddle-point problem \eqref{MatrixForm} can be written as: find $(\mathbf{u}_I,p_I,\mathbf{u}_\Gamma,p_0) \in (\mathbf{V}_I,Q_I,\mathbf{\widehat{V}}_\Gamma,Q_0)$, such that:
\begin{equation}\label{discMat}
	\left[
	\begin{array}{cccc}
		A_{II} & B_{II}^T & \widehat{A}_{\Gamma I}^T & 0\\
		B_{II} & 0 & \widehat{B}_{I\Gamma} & 0\\
		\widehat{A}_{\Gamma I} & \widehat{B}_{I\Gamma}^T & \widehat{A}_{\Gamma \Gamma} & \widehat{B}_{0\Gamma}^T\\
		0 & 0 & \widehat{B}_{0\Gamma}^T & 0\\
	\end{array}
	\right]
	\left[
	\begin{array}{c}
		\mathbf{u}_I\\
		p_I\\
		\mathbf{u}_\Gamma\\
		p_0\\
	\end{array}
	\right]
	=
	\left[
	\begin{array}{c}
		\mathbf{f}_I\\
		0\\
		\mathbf{f}_\Gamma\\
		0\\
	\end{array}
	\right].
\end{equation}
\begin{remark}
	Here the lower left block of \eqref{discMat}  is zero because the bilinear form $b(\mathbf{u}_I,q_0)$ vanishes for any $\mathbf{v}_I \in \mathbf{V}_I$ and $q_0 \in Q_0$.
	To keep this property, when the change of basis for the pressure space is applied, it is important to take care of the fact that the shape and dimension of the elements is different.
\end{remark} 

The blocks related to the continuous interface velocity are assembled from the corresponding subdomain submatrices, e.g., $\widehat{A}_{\Gamma \Gamma} = \sum_{i=1}^{N} {R_\Gamma^{(i)}}^T \widehat{A}_{\Gamma \Gamma}^{(i)} R_\Gamma^{(i)}$ and $\widehat{B}_{0 \Gamma} = \sum_{i=1}^{N} \widehat{B}_{0 \Gamma}^{(i)} R_\Gamma^{(i)}$. Correspondingly, the right-hand side vector $\mathbf{f}_I$ consists of subdomain vectors $\mathbf{f}_I^{(i)}$, and $\mathbf{f}_\Gamma$ is assembled from the subdomain components $\mathbf{f}_\Gamma^{(i)}$; we denote the spaces of the right-hand side vectors $\mathbf{f}_I$ and $\mathbf{f}_\Gamma$ by $\mathbf{F}_I$ and $\mathbf{F}_\Gamma$ respectively.

By employing a symmetric permutation, the leading two by two blocks in the coefficient matrix can be rewritten as a block diagonal matrix with blocks corresponding to independent subdomain problems. We show here how such a matrix takes form in the simplest case of two subdomains:
\begin{equation}\label{twosub}
	\left[
	\begin{array}{cccccc}
		A_{II}^{(1)} & {B_{II}^{(1)}}^T & 0 & 0 & {A_{\Gamma I}^{(1)}}^T & 0\\
		B_{II}^{(1)} & 0 & 0 & 0 & B_{I\Gamma}^{(1)} & 0\\
		0 & 0 & A_{II}^{(2)} & {B_{II}^{(2)}}^T & {A_{\Gamma I}^{(2)}}^T & 0\\
		0 & 0 & B_{II}^{(2)} & 0 & B_{I\Gamma}^{(2)} & 0\\
		A_{\Gamma I}^{(1)} & {B_{I\Gamma}^{(1)}}^T & A_{\Gamma I}^{(2)} & {B_{I\Gamma}^{(2)}}^T & \widehat{A}_{\Gamma\Gamma} & {\widehat{B}_{0\Gamma}}^T\\
		0 & 0 & 0 & 0 & \widehat{B}_{0\Gamma} & 0
	\end{array}
	\right]
	\left[
	\begin{array}{c}
		\mathbf{u}_I^{(1)}\\
		p_I^{(1)}\\
		\mathbf{u}_I^{(2)}\\
		p_I^{(2)}\\
		\mathbf{u}_\Gamma\\
		p_0\\
	\end{array}
	\right]
	=
	\left[
	\begin{array}{c}
		\mathbf{f}_I^{(1)}\\
		0\\
		\mathbf{f}_I^{(2)}\\
		0\\
		\mathbf{f}_\Gamma\\
		0\\
	\end{array}
	\right].
\end{equation}
In the rest of this section the computations are always performed in the case of two subdomains. The extension to the general case with more subdomains is natural, but the computations are clearly more involved.\\
We proceed eliminating, by static condensation, the independent subdomain variables $(\mathbf{u}_I^{(1)},p_I^{(1)})$ and $(\mathbf{u}_I^{(2)},p_I^{(2)})$ in the system (\ref{twosub}). To do so, we solve two independent Dirichlet problems: 
\begin{equation}
	\left[
	\begin{array}{cc}
		A_{II}^{(1)} & {B_{II}^{(1)}}^T\\
		B_{II}^{(1)} & 0
	\end{array}
	\right]
	\left[
	\begin{array}{c}
		\mathbf{u}_I^{(1)}\\
		p_I^{(1)}\\
	\end{array}
	\right] + 
	\left[
	\begin{array}{cc}
		{A_{\Gamma I}^{(1)}}^T & 0\\
		B_{I\Gamma}^{(1)} & 0
	\end{array}
	\right]
	\left[
	\begin{array}{c}
		\mathbf{u}_\Gamma\\
		p_0\\
	\end{array}
	\right] = 
	\left[
	\begin{array}{c}
		\mathbf{F}_I^{(1)}\\
		0\\
	\end{array}
	\right] \text{,}
\end{equation}
\begin{align}
	\left[
	\begin{array}{cc}
		A_{II}^{(2)} & {B_{II}^{(2)}}^T\\
		B_{II}^{(2)} & 0
	\end{array}
	\right]
	\left[
	\begin{array}{c}
		\mathbf{u}_I^{(2)}\\
		p_I^{(2)}\\
	\end{array}
	\right] + 
	\left[
	\begin{array}{cc}
		{A_{\Gamma I}^{(2)}}^T & 0\\
		B_{I\Gamma}^{(2)} & 0
	\end{array}
	\right]
	\left[
	\begin{array}{c}
		\mathbf{u}_\Gamma\\
		p_0\\
	\end{array}
	\right] = 
	\left[
	\begin{array}{c}
		\mathbf{F}_I^{(2)}\\
		0\\
	\end{array}
	\right]\text{,}
\end{align}
thus
\begin{equation}\label{i1}
	\left[
	\begin{array}{c}
		\mathbf{u}_I^{(1)}\\
		p_I^{(1)}\\
	\end{array}
	\right] = 
	\left[
	\begin{array}{cc}
		A_{II}^{(1)} & {B_{II}^{(1)}}^T\\
		B_{II}^{(1)} & 0
	\end{array}
	\right]^{-1}
	\left(
	\left[
	\begin{array}{c}
		\mathbf{F}_I^{(1)}\\
		0\\
	\end{array}
	\right] - 
	\left[
	\begin{array}{cc}
		{A_{\Gamma I}^{(1)}}^T & 0\\
		B_{I\Gamma}^{(1)} & 0
	\end{array}
	\right]
	\left[
	\begin{array}{c}
		\mathbf{u}_\Gamma\\
		p_0\\
	\end{array}
	\right]
	\right) \text{,}
\end{equation}
\begin{equation}\label{i2}
	\left[
	\begin{array}{c}
		\mathbf{u}_I^{(2)}\\
		p_I^{(2)}\\
	\end{array}
	\right] = 
	\left[
	\begin{array}{cc}
		A_{II}^{(2)} & {B_{II}^{(2)}}^T\\
		B_{II}^{(2)} & 0
	\end{array}
	\right]^{-1}
	\left(
	\left[
	\begin{array}{c}
		\mathbf{F}_I^{(2)}\\
		0\\
	\end{array}
	\right] - 
	\left[
	\begin{array}{cc}
		{A_{\Gamma I}^{(2)}}^T & 0\\
		B_{I\Gamma}^{(2)} & 0
	\end{array}
	\right]
	\left[
	\begin{array}{c}
		\mathbf{u}_\Gamma\\
		p_0\\
	\end{array}
	\right]
	\right) \text{,}
\end{equation}
Then, substituting the solutions of \eqref{i1} and \eqref{i2} in
\begin{equation}\label{eqGamma}
	\left[
	\begin{array}{cc}
		A_{\Gamma I}^{(1)} & {B_{I\Gamma}^{(1)}}^T \\
		0 & 0 
	\end{array}
	\right]
	\left[
	\begin{array}{c}
		\mathbf{u}_I^{(1)}\\
		p_I^{(1)}\\
	\end{array}
	\right] + 
	\left[
	\begin{array}{cc}
		A_{\Gamma I}^{(2)} & {B_{I\Gamma}^{(2)}}^T \\
		0 & 0 
	\end{array}
	\right]
	\left[
	\begin{array}{c}
		\mathbf{u}_I^{(2)}\\
		p_I^{(2)}\\
	\end{array}
	\right] + 
	\left[
	\begin{array}{cc}
		\widehat{A}_{\Gamma\Gamma} & {\widehat{B}_{0\Gamma}}^T\\
		\widehat{B}_{0\Gamma} & 0
	\end{array}
	\right]
	\left[
	\begin{array}{c}
		\mathbf{u}_\Gamma\\
		p_0\\
	\end{array}
	\right]
	=
	\left[
	\begin{array}{c}
		\mathbf{F}_\Gamma\\
		0\\
	\end{array}
	\right]
\end{equation}
we obtain the global interface saddle-point problem:
\begin{equation}\label{globInt}
	\widehat{S}\text{ }\widehat{u} = 
	\left[
	\begin{array}{cc}
		\widehat{S}_\Gamma & {\widehat{B}_{0\Gamma}}^T\\
		\widehat{B}_{0\Gamma} & 0\\
	\end{array}
	\right]
	\left[
	\begin{array}{c}
		\mathbf{u}_\Gamma\\
		p_0\\
	\end{array}
	\right]
	=
	\left[
	\begin{array}{c}
		\mathbf{g}_\Gamma\\
		0\\
	\end{array}
	\right]
	= \widehat{\mathbf{g}}\text{,}
\end{equation}
where the right-hand side $\widehat{\mathbf{g}}\in \mathbf{F}_\Gamma \times F_0$ is given by
\begin{align}
	\widehat{\mathbf{g}} = \sum_{i=1}^{2} {R_\Gamma^{(i)}}^T \bigg \{ 
	\left[
	\begin{array}{c}
		\mathbf{f}_\Gamma^{(i)}\\
		0\\
	\end{array}
	\right] - 
	\left[
	\begin{array}{cc}
		A_{\Gamma I}^{(i)} & {B_{I\Gamma}^{(i)}}^T\\
		0 & 0
	\end{array}
	\right]
	\left[
	\begin{array}{cc}
		A_{II}^{(i)} & {B_{II}^{(i)}}^T\\
		B_{II}^{(i)} & 0
	\end{array}
	\right]^{-1}
	\left[
	\begin{array}{c}
		\mathbf{f}_I^{(i)}\\
		0\\
	\end{array}
	\right]
	\bigg \}.
\end{align}
We note that $\widehat{S}$ is assembled from the subdomain Stokes Schur complements $S^{(i)}$, which are defined by: given ${\mathbf{w}}^{(i)} = \mathbf{w}_\Gamma^{(i)} \times q_0^{(i)}  \in \mathbf{V}_\Gamma^{(i)} \times Q_0^{(i)}$, determine $S^{(i)}{\mathbf{w}}^{(i)}\in \mathbf{F}_\Gamma^{(i)} \times F_0^{(i)}$ such that 
\begin{equation}\label{firstSchur}
	\left[
	\begin{array}{cccc}
		A_{II}^{(i)} & {B_{II}^{(i)}}^T & {A_{\Gamma I}^{(i)}}^T & 0\\
		B_{II}^{(i)} & 0 & B_{I\Gamma}^{(i)} & 0\\
		A_{\Gamma I}^{(i)} & {B_{I\Gamma}^{(i)}}^T & {A}_{\Gamma\Gamma}^{(i)} & {{B}_{0\Gamma}^{(i)}}^T\\
		0 & 0 & {B}_{0\Gamma}^{(i)} & 0
	\end{array}
	\right]
	\left[
	\begin{array}{c}
		\mathbf{w}_I^{(i)}\\
		q_I^{(i)}\\
		\mathbf{w}_\Gamma^{(i)}\\
		q_0^{(i)}
	\end{array}
	\right] = 
	\left[
	\begin{array}{c}
		\mathbf{0}\\ 0\\ S^{(i)}{\mathbf{w}}^{(i)}
	\end{array}
	\right]\text{.}
\end{equation}
Denoting by $S_\Gamma$ the direct sum of the $S^{(i)}_\Gamma$, then $\widehat{S}_\Gamma$ is given by
\begin{equation}\label{ScapGdef}
	\widehat{S}_\Gamma = R_\Gamma^T S_\Gamma R_\Gamma = \sum_{i=1}^{2} {R^{(i)}_\Gamma}^T S^{(i)}_\Gamma R^{(i)}_\Gamma,
\end{equation} 
and then we set 
\begin{equation}
	R = \left[
	\begin{array}{cc}
		R_\Gamma  & 0\\
		0 & I
	\end{array}
	\right] \text{,} \quad 
	R^{(i)} = \left[
	\begin{array}{cc}
		R_\Gamma^{(i)}  & 0\\
		0 & I
	\end{array}
	\right]\text{.}
\end{equation}
Finally we see from \eqref{firstSchur}, that the action of $S^{(i)}$ on a vector can be evaluated by solving a Dirichlet problem on the subdomain $\Omega_i$ as in (\ref{i1}) and (\ref{i2}),
so it is not necessary to assemble the matrix $\widehat{S}$ because only its action is required. \\
In the next section we introduce a BDDC preconditioner for problem \eqref{globInt}, where the operator of the preconditioned problem is symmetric and positive definite, so we will use the PCG method to solve it.

\subsection{BDDC preconditioner}
\label{sec:7}
We now present the BDDC preconditioner, first designed in \cite{li2006bddc} for finite element discretizations of the Stokes equations, that we will extend to the VEM discretization introduced in the previous sections.
This preconditioner is very similar to FETI-DP, but there is a main difference between them: while in a FETI-DP algorithm the continuity of the solution will not be fully satisfied until the algorithm has converged, in the BDDC one full continuity is restored at the end of each iteration step, by using an average operator. \\
Before entering into the definition of the function space used to construct the BDDC preconditioner, we briefly justify the choice of our notation.
The subscript $\Gamma$ indicates dofs living on the interface, $\Pi$ and $\Delta$ are instead used to distinguish dofs of $\Gamma$ that belong to the primal and dual spaces, respectively, defined here below. Two other subscripts are used: $C$ indicates an operator referred to the coarse space and $D$ is instead used to highlight that an operator has been rescaled by suitable scaling functions, defined later. The hat $\widehat{\cdot}$ refers to a continuous space, the $\widetilde{\cdot}$ means that the space is continuous on primal interface dofs and discontinuous on the dual ones and finally no hat is used for the product of local spaces, which is discontinuous at all interface dofs.\\
As a first step, we introduce a partially assembled interface velocity space $\mathbf{\widetilde{V}}_\Gamma$,
\begin{equation}
	\mathbf{\widetilde{V}}_\Gamma = \mathbf{\widehat{V}}_\Pi \bigoplus \mathbf{V}_\Delta = \mathbf{\widehat{V}}_\Pi \bigoplus \big( \prod_{i=1}^N \mathbf{V}_\Delta^{(i)} \big).
\end{equation}
$\mathbf{\widehat{V}}_\Pi$ is the continuous coarse level primal interface velocity space which typically is spanned by subdomain vertex nodal basis functions, and/or by interface edge basis functions with constant values, or with values of weight functions, on these edge. These basis functions correspond to the primal interface velocity continuity constraints, which will be discussed later. We will always assume that the basis has been changed so that each primal basis function corresponds to an explicit degree of freedom. In other words, we will have explicit primal unknowns corresponding to the primal continuity constraints on edges. The primal degrees of freedom are shared by neighboring subdomains. The complimentary space $\mathbf{V}_\Delta$ is the direct sum of the subdomain dual interface velocity spaces $\mathbf{V}_\Delta^{(i)}$ , which correspond to the remaining interface velocity degrees of freedom and are spanned by basis functions which vanish at the primal degrees of freedom. Thus, an element in the space $\mathbf{\widetilde{V}}_\Gamma$ has a continuous primal velocity and typically a discontinuous dual velocity component.\\
We now introduce several restriction, extension, and scaling operators between a variety of spaces. As in \cite{li2006bddc}, $R_\Gamma^{(i)}$ is the operator which maps a function in the space  $\mathbf{\widehat{V}}_\Gamma$ to its component in $\mathbf{V}_\Gamma^{(i)}$.
We define $R_{\Delta}^{(i)}$ as the operator which maps the space $\mathbf{\widehat{V}}_\Gamma$ to its dual component in the space $\mathbf{V}_\Delta^{(i)}$. $R_{\Gamma \Pi}$ is the restriction operator from the space $\mathbf{\widehat{V}}_\Gamma$ to its subspace $\mathbf{\widehat{V}}_\Pi$; $R_\Pi^{(i)}$ is the operator which maps $\mathbf{\widehat{V}}_\Pi$ into its $\Gamma_i$-component. $\widetilde{R}_\Gamma$ is the direct sum of $R_{\Gamma \Pi}$ and the $R_\Delta^{(i)}$ , and it is a map from $\mathbf{\widehat{V}}_\Gamma$ into $\mathbf{\widetilde{V}}_\Gamma$.\\
The relationships among the previous spaces and operators are summarized in the following diagram:
\[
\begin{tikzcd}
	& \mathbf{\widetilde{V}}_\Gamma & \\
    \mathbf{V}_\Gamma^{(i)}  & \ar{l}[swap]{R_\Gamma^{(i)}} \mathbf{\widehat{V}}_\Gamma \ar{r}{R_\Delta^{(i)}}\ar{d}{R_{\Pi \Gamma}} \ar{u}[swap]{\widetilde{R}_\Gamma} & \mathbf{V}_\Delta^{(i)}\\
    & \mathbf{\widehat{V}}_\Pi \ar{r}{R_\Pi^{(i)}} & \mathbf{\widehat{V}}_\Pi^{(i)}\\
\end{tikzcd}
\]

In order to define certain scaling operators, which will be used in the definition of the BDDC preconditioner, see \eqref{BDDCprec} , we introduce a positive scaling factor $\delta^\dagger_i(x)$ for the nodes on the interface $\Gamma_i$ of each subdomain $\Omega_i$. For the type of problem we will use in the numerical experiment (incompressible Stokes problems), we simply define the  $\delta^\dagger_i(x)$ as the pseudoinverse counting functions, so:
\begin{equation}\label{pseudoinv}
	\delta^\dagger_i(x):=1/card(I_x),\quad x\in \Gamma_{i} 
\end{equation}
where $I_x$ is the set of indices of subdomains which have $x$ on their boundaries and $card(I_x)$ is the number of these subdomains.
Now we can define the scaled restriction operators $R^{(i)}_{D,\Delta}$, simply multiplying each non-zero element of $R^{(i)}_{\Delta}$, only one for row, by the corresponding scaling factor $\delta^\dagger_i(x)$. We construct also the scaled operator $\widetilde{R}_{D,\Gamma}$ as the direct sum of $R_{\Gamma,\Pi}$ and $R^{(i)}_{D,\Delta}$.
After the change of basis, the interface velocity Schur complement $\widetilde{S}_\Gamma$ is defined on the partially assembled interface velocity space $\widetilde{\mathbf{V}}_\Gamma$ by: given $\mathbf{v}_\Gamma \in \widetilde{\mathbf{V}}_\Gamma$, $\widetilde{S}_\Gamma \mathbf{v}_\Gamma \in \mathbf{\widetilde{F}}_\Gamma$ satisfies\\

\begin{align}
	\left[
	\begin{array}{ccccc}
		A_{II}^{(i)} & B_{II}^{{(i)}^T} & A_{\Delta I}^{{(i)}^T} & & \widetilde{A}_{\Pi I}^{{(i)}^T}\\
		B_{II}^{(i)} & 0 & B_{I\Delta}^{(i)} & & \widetilde{B}_{I\Pi}^{(i)}\\
		A_{\Delta I}^{(i)} & B_{I \Delta}^{{(i)}^T} & A_{\Delta \Delta}^{(i)} & & \widetilde{A}_{\Pi \Delta}^{{(i)}^T}\\
		& & & \ddots & \vdots \\
		\widetilde{A}_{\Pi I}^{(i)} & \widetilde{B}_{I\Pi}^{{(i)}^T} & \widetilde{A}_{\Pi \Delta}^{(i)} & \ldots & \widetilde{A}_{\Pi \Pi}
	\end{array}
	\right]
	\left[
	\begin{array}{c}
		\mathbf{v}_I^{(i)}\\
		p_I^{(i)}\\
		\mathbf{v}_\Delta^{(i)}\\
		\vdots\\	
		\mathbf{v}_\Pi^{(i)}\\
	\end{array}
	\right] =
	\left[
	\begin{array}{c}
		\mathbf{0}\\
		0\\
		(\widetilde{S}_\Gamma\mathbf{v}_\Gamma)_\Delta^{(i)}\\
		\vdots\\	
		(\widetilde{S}_\Gamma\mathbf{v}_\Gamma)_\Pi^{(i)}\\
	\end{array}
	\right]
\end{align}
Here $\widetilde{A}_{\Pi\Pi} = \sum_{i=1}^{N}{R_\Pi^{(i)}}^T A_{\Pi\Pi}^{(i)} R_\Pi^{(i)}$, $\widetilde{A}_{\Pi I}^{(i)} = {R_\Pi^{(i)}}^T A_{\Pi I}^{(i)}$, $\widetilde{A}_{\Pi \Delta}^{(i)} = {R_\Pi^{(i)}}^T A_{\Pi \Delta}^{(i)}$ and $\widetilde{B}_{I \Pi}^{(i)} = B_{I \Pi}^{(i)} R_\Pi^{(i)}$.

Defining by $\overline{R}_\Gamma$ the operator that maps the space $\widetilde{\mathbf{V}}_\Gamma$ into the product space $\mathbf{V}_\Gamma$ associated with the set of subdomains, we observe that $\widetilde{S}_\Gamma$ can be obtained from the Schur complements $S^{(i)}_\Gamma$ by assembling only the primal interface velocity part, i.e. as 
\begin{equation}\label{stildedef}
	\widetilde{S}_\Gamma = {\overline{R}_\Gamma}^T S_\Gamma \overline{R}_\Gamma.
\end{equation}
As we saw before \eqref{firstSchur} the global interface Schur operator $\widehat{S}_\Gamma$ is obtanied by fully assembling the $S_\Gamma^{(i)}$ across the subdomain interface, therefore it can be also obtained from $\widetilde{S}_\Gamma$ by further assembling the dual interface velocity part, $\widehat{S}_\Gamma = {\widetilde{R}_\Gamma}^T \widetilde{S}_\Gamma \widetilde{R}_\Gamma$. So we need to define an operator $\widetilde{B}_{0\Gamma}$, which maps the partially assembled interface velocity space $\widetilde{V}_\Gamma$ into $F_0$, the space of right hand sides corresponding to $Q_0$, and it is obtained from $\widetilde{B}_{0\Gamma}$ by assembling the dual interface velocity part on the subdomain interfaces, i.e. $\widehat{B}_{\Gamma} = \widetilde{B}_{\Gamma} \widetilde{R}_\Gamma$.\\ 
Introducing
\begin{align}
	\widetilde{R} = \left[
	\begin{array}{cc}
		\widetilde{R}_\Gamma  & 0\\
		0 & I
	\end{array}
	\right] \text{,} \quad 
	\widetilde{S} = \left[
	\begin{array}{cc}
		\widetilde{S}_\Gamma  & \widetilde{B}_{0\Gamma}^T\\
		\widetilde{B}_{0\Gamma} & 0
	\end{array}
	\right]\text{,}
\end{align}
we can write $\widehat{S}$, the operator of the global interface problem \eqref{globInt}, as 
\begin{align}
	\widehat{S} = \left[
	\begin{array}{cc}
		\widehat{S}_\Gamma  & {\widehat{B}_{0\Gamma}}^T\\
		\widehat{B}_{0\Gamma} & 0
	\end{array}
	\right] = 
	\left[
	\begin{array}{cc}
		{\widetilde{R}_\Gamma}^T \widetilde{S}_\Gamma \widetilde{R}_\Gamma   & {\widetilde{R}_\Gamma}^T{\widetilde{B}_{0\Gamma}}^T\\
		\widetilde{B}_{0\Gamma}\widetilde{R}_\Gamma & 0
	\end{array}
	\right] = \widetilde{R}^T \widetilde{S} \widetilde{R}\text{.}
\end{align}
The preconditioner for solving the global saddle-point problem \eqref{globInt} is
\begin{equation}\label{BDDCprec}
	M^{-1}=\widetilde{R}_D^T \widetilde{S}^{-1} \widetilde{R}_D,
\end{equation}
where we have defined 
\begin{align}
	\widetilde{R}_D:= \left[
	\begin{array}{cc}
		\widetilde{R}_{D,\Gamma}  & 0\\
		0 & I
	\end{array}
	\right] \text{,}
\end{align}
and so we have the BDDC preconditioned problem: find $(\mathbf{u_\Gamma},p_0)\in \mathbf{\widehat{V}}_\Gamma \times Q_0$, such that 
\begin{align}
	\widetilde{R}_D^T \widetilde{S}^{-1} \widetilde{R}_D \widehat{S}
	\left[
	\begin{array}{c}
		\mathbf{u}_\Gamma\\
		p_0\\
	\end{array}
	\right]
	= \widetilde{R}_D^T \widetilde{S}^{-1} \widetilde{R}_D
	\left[
	\begin{array}{c}
		\mathbf{g}_\Gamma\\
		0\\
	\end{array}
	\right]\text{.}
\end{align}
What we need in our implementation is to determine the action $\widetilde{S}^{-1}\mathbf{q}$ for any given $\mathbf{q} = (\mathbf{q_\Gamma},q_0) \in \widetilde{\mathbf{F}}_\Gamma \times F_0$, so we have to solve the linear system
\begin{equation}\label{bo}
	\left[
	\begin{array}{cc}
		\widetilde{S}_\Gamma & \widetilde{B}_{0\Gamma}^T\\
		\widetilde{B}_{0\Gamma} & 0\\
	\end{array}
	\right]
	\left[
	\begin{array}{c}
		\mathbf{u}_\Gamma\\
		p_0\\
	\end{array}
	\right]
	=
	\left[
	\begin{array}{c}
		\mathbf{q}_\Gamma\\
		q_0\\
	\end{array}
	\right]\text{.}
\end{equation}
Given the definition of $\widetilde{S}_\Gamma$ in (5.3), we have that solving (5.10) is equivalent to solve

\begin{align}
	\left[
	\begin{array}{cccccc}
		A_{II}^{(i)} & B_{II}^{{(i)}^T} & A_{\Delta I}^{{(i)}^T} & & \widetilde{A}_{\Pi I}^{{(i)}^T} & 0 \\
		B_{II}^{(i)} & 0 & B_{I\Delta}^{(i)} & & \widetilde{B}_{I\Pi}^{(i)} & 0 \\
		A_{\Delta I}^{(i)} & B_{I \Delta}^{{(i)}^T} & A_{\Delta \Delta}^{(i)} & & \widetilde{A}_{\Pi \Delta}^{{(i)}^T} & B_{0 \Delta}^{{(i)}^T}\\
		& & & \ddots & \vdots & \\
		\widetilde{A}_{\Pi I}^{(i)} & \widetilde{B}_{I\Pi}^{{(i)}^T} & \widetilde{A}_{\Pi \Delta}^{(i)} & \ldots & \widetilde{A}_{\Pi \Pi} & \widetilde{B}_{0 \Pi}^{T}\\
		0 & 0 & B_{0 \Delta}^{(i)} & & \widetilde{B}_{0 \Pi} & 0 \\
	\end{array}
	\right]
	\left[
	\begin{array}{c}
		\mathbf{u}_I^{(i)}\\
		p_I^{(i)}\\
		\mathbf{u}_\Delta^{(i)}\\
		\vdots\\	
		\mathbf{u}_\Pi\\
		p_0\\
	\end{array}
	\right] =
	\left[
	\begin{array}{c}
		\mathbf{0}\\
		0\\
		\mathbf{q}_\Delta^{(i)}\\
		\vdots\\	
		\mathbf{q}_\Pi \\
		q_0\\
	\end{array}
	\right]
\end{align}
where $\widetilde{B}_{0\Pi} = \sum_{i=1}^{N} B_{0\Pi}^{(i)} R_\Pi^{(i)}$. 
Now using a block factorization we obtain
\begin{align}
	\widetilde{S}^{-1} = \sum_{i=1}^{N} 
	\left[
	\begin{array}{ccc}
		0 & 0 & R_{\Delta,i}^T
	\end{array}
	\right]
	\left[
	\begin{array}{ccc}
		A_{II}^{(i)} & B_{II}^{{(i)}^T} & A_{\Delta I}^{{(i)}^T}\\
		B_{II}^{(i)} & 0 & B_{I \Delta}^{(i)}\\
		A_{\Delta I}^{(i)} & B_{I \Delta}^{{(i)}^T} & A_{\Delta \Delta}^{(i)}\\
	\end{array}
	\right]^{-1}
	\left[
	\begin{array}{c}
		0 \\
		0 \\
		R_{\Delta,i}
	\end{array}
	\right] + \Phi S_{CC}^{-1} \Phi^T \text{,}
\end{align}
where $R_{\Delta,i}$ maps $\widetilde{\mathbf{F}}_\Gamma \times F_0$ into $\mathbf{F}_\Delta^{(i)}$, the set of right hand sides corresponding to $\mathbf{V}_\Delta^{(i)}$.
The matrix $S_{CC}$, relatively to the primal constraints, has to be completely assembled in this way
\begin{align}
	\begin{split}
		S_{CC} = \sum_{i=1}^{N} R_C^{{(i)}^T} \bigg\{
		\left[
		\begin{array}{cc}
			A_{\Pi \Pi}^{(i)} & B_{0 \Pi}^{{(i)}^T} \\
			B_{0 \Pi}^{(i)} & 0
		\end{array}
		\right] - 
		\left[
		\begin{array}{ccc}
			A_{\Pi I}^{(i)} & B_{I \Pi}^{{(i)}^T} &  A_{\Pi \Delta}^{(i)} \\
			0 & 0 & B_{0 \Delta}^{(i)}
		\end{array}
		\right] \\
		\left[
		\begin{array}{ccc}
			A_{II}^{(i)} & B_{II}^{{(i)}^T} & A_{\Delta I}^{{(i)}^T}\\
			B_{II}^{(i)} & 0 & B_{I \Delta}^{(i)}\\
			A_{\Delta I}^{(i)} & B_{I \Delta}^{{(i)}^T} & A_{\Delta \Delta}^{(i)}\\
		\end{array}
		\right]^{-1}
		\left[
		\begin{array}{cc}
			A_{\Pi I}^{{(i)}^T} & 0 \\
			B_{I\Pi}^{(i)} & 0 \\
			A_{\Pi \Delta}^{{(i)}^T} & B_{0 \Delta}^{{(i)}^T} \\
		\end{array}
		\right]
		\bigg\} R_C^{(i)} \text{,}
	\end{split}
\end{align}
where we have defined 
\begin{align}
	R_C^{(i)} := \left[
	\begin{array}{cc}
		R_\Pi^{(i)} & 0\\
		0 & I
	\end{array}\right] \text{,}
\end{align}
the maps from $\widehat{\mathbf{V}}_\Pi \times Q_0$ to ${\mathbf{V}}_\Pi^{(i)} \times Q_0$. 
Finally we define the matrix 
\begin{align}
	\Phi = R_{\Pi 0}^T - \sum_{i=1}^{N} \left[
	\begin{array}{ccc}
		0 & 0 & R_{\Delta,i}^T
	\end{array}
	\right]
	\left[
	\begin{array}{ccc}
		A_{II}^{(i)} & B_{II}^{{(i)}^T} & A_{\Delta I}^{{(i)}^T}\\
		B_{II}^{(i)} & 0 & B_{I \Delta}^{(i)}\\
		A_{\Delta I}^{(i)} & B_{I \Delta}^{{(i)}^T} & A_{\Delta \Delta}^{(i)}\\
	\end{array}
	\right]^{-1}
	\left[
	\begin{array}{cc}
		A_{\Pi I}^{{(i)}^T} & 0 \\
		B_{I\Pi}^{(i)} & 0 \\
		A_{\Pi \Delta}^{{(i)}^T} & B_{0 \Delta}^{{(i)}^T} \\
	\end{array}
	\right] R_C^{(i)} \text{,}
\end{align}
where $R_{\Pi 0}$ is the map between the space $\widetilde{\mathbf{F}}_\Gamma \times F_0$ and $\widehat{\mathbf{F}}_\Pi \times F_0$.
\section{Theoretical estimates}
\label{sec:8}
We now present an estimate for the eigenvalues of the preconditioned operator $M^{-1}\widehat{S}$, following the theory developed in \cite{li2006bddc} and adapting it to our VEM formulation. We can do so because the space $\mathbf{V}_\Gamma$ coincides with the analogous space that would be obtained applying the same procedure with the FEM.  
This substantially allow us to carry over the theory formulated for the FEM to the VEM, except for the second assumption we will see later. For this result, a different proof is necessary and we follow \cite{bertoluzza2017bddc}, where a proof independent of the tassellation is given. 

We have, as a consequence of a result on the inertia of Schur complements, the following:
\begin{lemma}
	The subdomain Schur complements $S_\Gamma^{(i)}$, defined in \eqref{globInt}, are symmetric and positive definite.
\end{lemma}
\begin{proof}
We know from (\ref{firstSchur}) that the Schur complement related to the velocity is defined by: given $\mathbf{w}_\Gamma^{(i)} \in \mathbf{V}_\Gamma^{(i)}$, determine $S_\Gamma^{(i)}{\mathbf{w}_\Gamma}^{(i)}\in \mathbf{F}_\Gamma^{(i)}$ such that 
\begin{equation}\label{secondSchur}
	\left[
	\begin{array}{ccc}
		A_{II}^{(i)} & {B_{II}^{(i)}}^T & {A_{\Gamma I}^{(i)}}^T\\
		B_{II}^{(i)} & 0 & B_{I\Gamma}^{(i)}\\
		A_{\Gamma I}^{(i)} & {B_{I\Gamma}^{(i)}}^T & {A}_{\Gamma\Gamma}^{(i)}\\
	\end{array}
	\right]
	\left[
	\begin{array}{c}
		\mathbf{w}_I^{(i)}\\
		q_I^{(i)}\\
		\mathbf{w}_\Gamma^{(i)}
	\end{array}
	\right] = 
	\left[
	\begin{array}{c}
		\mathbf{0}\\ 0\\ S_\Gamma^{(i)}{\mathbf{w}_\Gamma}^{(i)}
	\end{array}
	\right]\text{.}
\end{equation}
By the coercivity of $a(\cdot,\cdot)$ we know that the matrices 
\begin{align*}
\left[
	\begin{array}{cc}
		A_{II}^{(i)} & {A_{\Gamma I}^{(i)}}^T\\
		A_{\Gamma I}^{(i)}  & {A}_{\Gamma\Gamma}^{(i)}\\
	\end{array}
	\right]
\end{align*}
are symmetric and positive definite and so the left two by two upper block of the left-hand-side of (\ref{secondSchur}) has the same number of negative eigenvalues of the all matrix. 
Now, the left-hand-side matrices of (\ref{secondSchur}) are congruent to: 
\begin{align*}
	\left[
	\begin{array}{cccc}
		A_{II}^{(i)} & {B_{II}^{(i)}}^T & 0\\
		B_{II}^{(i)} & 0 & 0\\
		0 & 0 & S_{\Gamma}^{(i)}\\
	\end{array}
	\right]
\end{align*}
and so, by the Sylvester's law of inertia the velocity Schur complements are positive definite.
\end{proof}
In the following, we denote by $a^{(i)}$, $a_h^{(i)}$ and $b^{(i)}$ the restrictions to subdomain $\Omega_i$ of the
bilinear forms $a$, $a_h$ and $b$, respectively.
Then, we introduce the $|.|_{S_\Gamma^{(i)}}$ and $|.|_{S_\Gamma}$ seminorms defined by
\begin{align}
	|\mathbf{v}_\Gamma^{(i)}|_{S_\Gamma^{(i)}}^2 = {\mathbf{v}_\Gamma^{(i)}}^TS_\Gamma^{(i)} \mathbf{v}_\Gamma^{(i)}\text{,} \quad |\mathbf{v}_\Gamma|_{S_\Gamma}^2 = {\mathbf{v}_\Gamma}^T S_\Gamma \mathbf{v}_\Gamma = \sum_{i=1}^{N} |\mathbf{v}_\Gamma^{(i)}|_{S_\Gamma^{(i)}}^2,
\end{align}
and a norm and a seminorm on the space $\mathbf{V}_\Gamma^{(i)}$ 
\begin{equation}\label{semigrad}
	\Vert\mathbf{v}_\Gamma^{(i)}\Vert_{1/2,\Gamma_i}^2 = \Vert \mathbf{v}_\Gamma^{(i)}\Vert_{[H^{1/2}(\partial \Omega_i)]^2}^2\text{,}\qquad |\mathbf{v}_\Gamma^{(i)}|_{1/2,\Gamma_i}^2 = | \mathbf{v}_\Gamma^{(i)}|_{[H^{1/2}(\partial \Omega_i)]^2}^2\text{,}
\end{equation}
with consequently the norm $\Vert.\Vert_{1/2,\Gamma}$ and seminorm $|.|_{1/2,\Gamma}$ defined on the space $\mathbf{V}_\Gamma$ by $\Vert\mathbf{v}_\Gamma\Vert_{1/2,\Gamma}^2 = \sum_{i=1}^{N} \Vert\mathbf{v}_\Gamma^{(i)}\Vert_{1/2,\Gamma_i}^2$ and $|\mathbf{v}_\Gamma|_{1/2,\Gamma}^2 = \sum_{i=1}^{N} |\mathbf{v}_\Gamma^{(i)}|_{1/2,\Gamma_i}^2$.
\begin{lemma}\label{Bramble}
	There exist positive constant $c_1$ and $c_2$, independent of $H$, $h$ and the shape of subdomains, such that
	\begin{equation*}
		c_1\tilde{\beta}^2|\mathbf{v}_\Gamma|_{S_\Gamma}^2 \leq |\mathbf{v}_\Gamma|_{1/2,\Gamma}^2 \leq c_2 |\mathbf{v}_\Gamma|_{S_\Gamma}^2 \quad \forall \mathbf{v}_\Gamma \in \mathbf{V}_\Gamma\text{,}
	\end{equation*}
	where $\tilde{\beta}$ is the inf-sup stability constant defined in \eqref{infsupdisc}.
\end{lemma}
\begin{proof}
	This proof follows substantially the result presented in Bramble and Pasciak (\cite{bramble1990domain} Theorem 4.1) where a proof for FEM is provided.
	Given $\mathbf{v}_\Gamma \in \mathbf{V}_\Gamma$, we define the operators $T:\mathbf{V}_\Gamma \rightarrow \mathbf{\widehat{V}}$ and $S:\mathbf{V}_\Gamma \rightarrow Q$ satisfying $\forall i=1,...,N$:
	\begin{align}
		\begin{split}
			&(1) \quad S(\mathbf{v}_\Gamma)|_{\Omega_i} \in Q_I^{(i)},\\
			&(2) \quad T(\mathbf{v}_\Gamma)|_\Gamma = \mathbf{v}_\Gamma,\\
			&(3) \quad a_h^{(i)}(T(\mathbf{v}_\Gamma),\mathbf{v})+b^{(i)}(\mathbf{v},S(\mathbf{v}_\Gamma) = \mathbf{0} \quad \text{for all} \quad \mathbf{v} \in \mathbf{V}_I^{(i)},\\
			&(4) \quad b^{(i)}(T(\mathbf{v}_\Gamma), q) = 0 \quad \text{for all} \quad q \in Q_I^{(i)}.\\
		\end{split}
	\end{align}
	The above condition uniquely defines $S$ and $T$.
	Now given $\mathbf{v}_\Gamma\in \mathbf{V}_\Gamma$, let $\mathbf{v}_\Gamma^\mathcal{H}\in \mathbf{\widehat{V}}$ be the discrete harmonic extension of $\mathbf{v}_\Gamma$, i.e. the unique function in $\mathbf{\widehat{V}}$ which equals $\mathbf{v}_\Gamma$ on $\Gamma$ and 
	satisfies $\forall i=1,...,N$:
	\begin{align}
		a^{(i)}(\mathbf{v}_\Gamma^\mathcal{H},\mathbf{v}) = 0 \quad \text{for all} \quad \mathbf{v} \in \mathbf{V}_I^{(i)}.
	\end{align}
	By the stability of the discrete harmonic extension and the stability of the discrete bilinear form $a_h$ \cite{da2017divergence}, we have on each subdomain:
	\begin{align}\label{h1}
		a_h^{(i)}(\mathbf{v}_\Gamma^\mathcal{H},\mathbf{v}_\Gamma^\mathcal{H})\leq c_3 a^{(i)}(\mathbf{v}_\Gamma^\mathcal{H},\mathbf{v}_\Gamma^\mathcal{H})\leq c_3|\mathbf{v}^{(i)}_\Gamma|^2_{1/2,\Gamma_i},
	\end{align}
	where $c_3$ is a positive constant independent of $h$, $H$ and the number of subdomains $N$.
	Now, by definition of $S$ and $T$, and since $b^{(i)}(T(\mathbf{v}_\Gamma),S(\mathbf{v}_\Gamma)) = 0$, we have:
	\begin{align}\label{bes}
		a_h^{(i)}(T(\mathbf{v}_\Gamma),T(\mathbf{v}_\Gamma)) = a_h^{(i)}(T(\mathbf{v}_\Gamma),\mathbf{v}_\Gamma^\mathcal{H}) + b^{(i)}(\mathbf{v}_\Gamma^\mathcal{H},S(\mathbf{v}_\Gamma)).
	\end{align}
	Applying \eqref{infsupdisc} on the subdomains, we have, for some $c > 0$:
	\begin{align}\label{bor}
		\begin{split}
			||S(\mathbf{v}_\Gamma)||^2_{Q_i} \leq \tilde{\beta}^{-2} \sup_{\mathbf{w}\in \mathbf{V}_I^{(i)}}\frac{b^{(i)}(\mathbf{w},S(\mathbf{v}_\Gamma))^2}{\Vert \mathbf{w} \Vert_1^2} \leq c \tilde{\beta}^{-2} \sup_{\mathbf{w}\in \mathbf{V}_I^{(i)}}\frac{b^{(i)}(\mathbf{w},S(\mathbf{v}_\Gamma))^2}{a_h^{(i)}(\mathbf{w},\mathbf{w})}\\
			= c \tilde{\beta}^{-2} \sup_{\mathbf{w}\in \mathbf{V}_I^{(i)}}\frac{a_h^{(i)}(T(\mathbf{v}_\Gamma),\mathbf{w})^2}{a_h^{(i)}(\mathbf{w},\mathbf{w})} \leq c \tilde{\beta}^{-2} a_h^{(i)}(T(\mathbf{v}_\Gamma),T(\mathbf{v}_\Gamma)),
		\end{split}
	\end{align}
	Applying Cauchy-Schwarz to the first term in \eqref{bes} and using \eqref{bor}, we have:
	\begin{align*}
		\begin{split}
			|\mathbf{v}^{(i)}_\Gamma|^2_{S_\Gamma^{(i)}} = a_h^{(i)}(T(\mathbf{v}_\Gamma),T(\mathbf{v}_\Gamma)) \leq a_h^{(i)}(T(\mathbf{v}_\Gamma),T(\mathbf{v}_\Gamma))^{1/2}a_h^{(i)}(\mathbf{v}_\Gamma^\mathcal{H},\mathbf{v}_\Gamma^\mathcal{H})^{1/2}\\ + c|\mathbf{v}_\Gamma^\mathcal{H}|_{H^1(\Omega_i)}||S(\mathbf{v}_\Gamma)||_{Q_i}
			\leq a_h(T(\mathbf{v}_\Gamma),T(\mathbf{v}_\Gamma))^{1/2}a_h(\mathbf{v}_\Gamma^\mathcal{H},\mathbf{v}_\Gamma^\mathcal{H})^{1/2} +\\ c\tilde{\beta}^{-1}a_h(\mathbf{v}_\Gamma^\mathcal{H},\mathbf{v}_\Gamma^\mathcal{H})^{1/2}  a_h^{(i)}(T(\mathbf{v}_\Gamma),T(\mathbf{v}_\Gamma))^{1/2}
		\end{split}
	\end{align*}
	and then:
	\begin{equation}
		c \tilde{\beta}^2 |\mathbf{v}^{(i)}_\Gamma|^2_{S_\Gamma^{(i)}} = c \tilde{\beta}^2 a_h^{(i)}(T(\mathbf{v}_\Gamma),T(\mathbf{v}_\Gamma)) \leq a_h^{(i)}(\mathbf{v}_\Gamma^\mathcal{H},\mathbf{v}_\Gamma^\mathcal{H}).
	\end{equation}
	Finally, from \eqref{h1} and summing on the subdomains
	\begin{align*}
		c \tilde{\beta}^2 |\mathbf{v}_\Gamma|_{S_\Gamma}^2= c \tilde{\beta}^2 a_h(T(\mathbf{v}_\Gamma),T(\mathbf{v}_\Gamma)) \leq a_h(\mathbf{v}_\Gamma^\mathcal{H},\mathbf{v}_\Gamma^\mathcal{H}) \leq c_3 |\mathbf{v}_\Gamma|_{1/2,\Gamma}^2,
	\end{align*}
	which yields the first inequality of the thesis with $c_1=c/c_3$.
	
	For the second inequality we have, by definition of the discrete harmonic extension and again the stability of the discrete bilinear form $a_h$:
	\begin{equation}
		|\mathbf{v}_\Gamma|_{1/2,\Gamma_i}^2 \leq a^{(i)}(\mathbf{v}_\Gamma^\mathcal{H},\mathbf{v}_\Gamma^\mathcal{H}) \leq a^{(i)}(T(\mathbf{v}_\Gamma),T(\mathbf{v}_\Gamma)) \leq c_2 a_h^{(i)}(T(\mathbf{v}_\Gamma),T(\mathbf{v}_\Gamma))
	\end{equation}
	and then:
	\begin{equation}
		|\mathbf{v}_\Gamma|_{1/2,\Gamma}^2 \leq c_2 |\mathbf{v}_\Gamma|_{S_\Gamma}^2 \text{.} 
	\end{equation}
\end{proof}

The operators $\widehat{S}_\Gamma$ and $\widetilde{S}_\Gamma$, given in \eqref{ScapGdef} and \eqref{stildedef}, are both symmetric and positive definite, because of the Dirichlet boundary conditions on $\partial \Omega$ and provided that sufficiently many primal constraints are chosen. We can then define the $\widehat{S}_\Gamma$ and $\widetilde{S}_\Gamma$ norms on the spaces $\widehat{V}_\Gamma$ and $\widetilde{V}_\Gamma$ by 
\begin{align*}
	\begin{split}
		\Vert \mathbf{v}_\Gamma \Vert_{\widehat{S}_\Gamma}^2 = \mathbf{v}_\Gamma^T R_\Gamma^T S_\Gamma R_\Gamma \mathbf{v}_\Gamma = |R_\Gamma \mathbf{v}_\Gamma|_{S_\Gamma}^2 \quad \forall \mathbf{v}_\Gamma \in \widehat{\mathbf{V}}_\Gamma\text{,}\\
		\Vert \mathbf{v}_\Gamma \Vert_{\widetilde{S}_\Gamma}^2 = \mathbf{v}_\Gamma^T \overline{R}_\Gamma^T S_\Gamma \overline{R}_\Gamma \mathbf{v}_\Gamma = |\overline{R}_\Gamma \mathbf{v}_\Gamma|_{S_\Gamma}^2 \quad \forall \mathbf{v}_\Gamma \in \widetilde{\mathbf{V}}_\Gamma\text{.}
	\end{split}
\end{align*}
We then define two spaces, whose utility is that, restricted to such spaces, the interface problem operators $\widehat{S}$ of \eqref{globInt} and $\widetilde{S}$ of \eqref{bo} are positive semi-definite. As in \cite{li2006bddc}, we give the following:
\begin{definition}
	Given the discrete spaces $\widehat{\mathbf{V}}_{\Gamma}$ and $\widetilde{\mathbf{V}}_{\Gamma}$, we define the two subspaces
	\begin{align*}
		\begin{split}
			\widehat{\mathbf{V}}_{\Gamma,B} = \{\mathbf{v}_\Gamma \in \widehat{\mathbf{V}}_\Gamma | \widehat{B}_{0\Gamma}\mathbf{v}_\Gamma = 0\},\\
			\widetilde{\mathbf{V}}_{\Gamma,B} = \{\mathbf{v}_\Gamma \in \widetilde{\mathbf{V}}_\Gamma | \widetilde{B}_{0\Gamma}\mathbf{v}_\Gamma = 0\}.
		\end{split}
	\end{align*}
	We call $\widehat{\mathbf{V}}_{\Gamma,B} \times Q_0$ and $\widetilde{\mathbf{V}}_{\Gamma,B} \times Q_0$ the benign subspaces of $\widehat{\mathbf{V}}_{\Gamma} \times Q_0$ and $\widetilde{\mathbf{V}}_{\Gamma} \times Q_0$. 
\end{definition} 
\begin{lemma}
	The interface operator $\widehat{S}$ of \eqref{globInt}, restricted to the subspace $\widehat{\mathbf{V}}_{\Gamma,B} \times Q_0$ is positive semi-definite. The same is true for $\widetilde{S}$ of \eqref{stildedef} restricted to $\widetilde{\mathbf{V}}_{\Gamma,B} \times Q_0$.
\end{lemma}
We define the $\widehat{S}$ and $\widetilde{S}$ seminorms on the benign subspaces
\begin{align*}
	\begin{split}
		|\mathbf{v}|_{\widehat{S}}^2 = \mathbf{v}^T \widehat{S} \mathbf{v} = \Vert \mathbf{v}_\Gamma\Vert_{\widehat{S}_\Gamma} \quad \forall \mathbf{v} = (\mathbf{v}_{\Gamma},q_0) \in \widehat{\mathbf{V}}_{\Gamma,B}\times Q_0\text{,}\\
		|\mathbf{v}|_{\widetilde{S}}^2 = \mathbf{v}^T \widetilde{S} \mathbf{v} = \Vert \mathbf{v}_\Gamma\Vert_{\widetilde{S}_\Gamma} \quad \forall \mathbf{v} = (\mathbf{v}_{\Gamma},q_0) \in \widetilde{\mathbf{V}}_{\Gamma,B}\times Q_0\text{.}
	\end{split}
\end{align*}
Now we define an average operator $E_D = \widetilde{R}\widetilde{R}_D^T$, which maps $\mathbf{\widetilde{V}}_\Gamma\times Q_0$, with generally discontinuous interface velocities, to elements with continuous interface velocities in the same space. For any $\mathbf{v} = (\mathbf{v}_\Gamma,q_0) \in \widetilde{\mathbf{V}}_{\Gamma,B}\times Q_0$,
\begin{align}
	E_D = 
	\left[
	\begin{array}{c}
		\mathbf{v}_\Gamma\\
		q_0
	\end{array}
	\right] = 
	\left[
	\begin{array}{cc}
		\widetilde{R}_\Gamma & 0\\
		0 & I
	\end{array}
	\right]
	\left[
	\begin{array}{cc}
		\widetilde{R}_{D,\Gamma} & 0\\
		0 & I
	\end{array}
	\right]
	\left[
	\begin{array}{c}
		\mathbf{v}_\Gamma\\
		q_0
	\end{array}
	\right] = 
	\left[
	\begin{array}{c}
		E_{D,\Gamma}\mathbf{v}_\Gamma\\
		q_0
	\end{array}
	\right]
\end{align} 
where $E_D = \widetilde{R}\widetilde{R}_{D,\Gamma}^T$, provides the average of the interface velocities across the interface $\Gamma$. Recalling that we can split $\mathbf{v} = \mathbf{v}_\Pi \oplus \mathbf{v}_\Delta$, we have $E_D \mathbf{v} = \mathbf{v}_\Pi \oplus E_{D,\Delta} \mathbf{v}_\Delta$, where $E_{D,\Delta} \mathbf{v}_\Delta$ is the dual part of the averaged vector.
As in the FEM case (see \cite{li2006bddc}) we need two assumptions to proceed in the discussion, these will be satisfied when a reasonable choice of the primal constraints will be done.
\begin{ass}\label{ass1}
	For any $\mathbf{v}_\Delta\in\mathbf{V}_\Delta$, $\int_{\partial\Omega_i}\mathbf{v}_\Delta^{(i)}\cdot\mathbf{n} = 0$ and $\int_{\partial\Omega_i}(E_{D,\Delta}\mathbf{v}_\Delta)^{(i)}\cdot\mathbf{n} = 0$, where $\mathbf{n}$ is the outward normal of $\partial\Omega_i$. We can equivalently write $B_{0\Delta}^{(i)}\mathbf{v}_\Delta^{(i)} = 0$ and $B_{0\Delta}^{(i)}(E_{D,\Delta}\mathbf{v}_\Delta)^{(i)} = 0$
\end{ass}
\begin{ass}\label{ass2}
	There exists a positive constant C, which is independent of $H$, $h$ and the number of subdomains, such that
	\begin{align*}
		|\bar{R}_\Gamma(E_{D,\Gamma}\mathbf{v}_\Gamma)|_{1/2,\Gamma} \leq C \bigg(1+\log\left(\frac{Hk^2}{h}\right)\bigg)|\bar{R}_\Gamma\mathbf{v}_\Gamma|_{1/2,\Gamma}\text{,} \quad \forall\mathbf{v}_\Gamma \in \mathbf{V}_\Gamma.
	\end{align*} 
\end{ass}
With these two assumptions, we have the following results (proof of \ref{lemma4} in \cite{li2006bddc}):
\begin{lemma}\label{lemma4}
	Let Assumption 1 hold. Then $\widetilde{R}_D^T\mathbf{v}\in \widehat{\mathbf{V}}_{\Gamma,B}\times Q_0$, for any $\mathbf{v}\in \widetilde{\mathbf{V}}_{\Gamma,B}\times Q_0$.
\end{lemma}
\begin{lemma}\label{lemma5}
	Let Assumptions 1 and 2 hold. Then there exists a positive constant C, which is independent of $H$, $h$ and the number of subdomains, such that 
	\begin{align*}
		|E_D\mathbf{v}|_{\widetilde{S}} \leq C \frac{1}{\tilde{\beta}}\bigg(1+\log\left(\frac{Hk^2}{h}\right)\bigg)|\mathbf{v}|_{\widetilde{S}}, \quad \forall \mathbf{v}=(\mathbf{v}_{\Gamma},q_0)\in \widetilde{\mathbf{V}}_{\Gamma,B}\times Q_0\text{,}
	\end{align*}
	where $\tilde{\beta}$ is the inf-sup stability constant of \eqref{infsupdisc}.
\end{lemma}
\begin{proof}
	Given any $\mathbf{v} = (\mathbf{v}_\Gamma,q_0) \in \widetilde{\mathbf{V}}_{\Gamma,B}\times Q_0$, we know, from Lemma \ref{lemma4}, that $\widetilde{R}_D^T\mathbf{v}\in \widehat{\mathbf{V}}_{\Gamma,B}\times Q_0$. Therefore, $E_D \mathbf{v} = \widetilde{R} \widetilde{R}_D^T\mathbf{v}\in \widetilde{\mathbf{V}}_{\Gamma,B}\times Q_0$. We have from the definition of the $\widetilde{S}$-seminorm, that
	\begin{equation}
		\vert E_D \mathbf{v} \vert^2_{\widetilde{S}} = \Vert E_{D,\Gamma}\mathbf{v}_\Gamma\Vert^2_{\widetilde{S}_\Gamma} = \vert \bar{R}_\Gamma(E_{D,\Gamma}\mathbf{v}_\Gamma)\vert^2_{S_\Gamma} \leq C\frac{1}{\tilde{\beta}^2}\vert\bar{R}_\Gamma(E_{D,\Gamma}\mathbf{v}_\Gamma)\vert^2_{1/2,\Gamma}\text{,}
	\end{equation}
	where the last inequality follows from Lemma \ref{Bramble}.
	We have, from Assumption 2 and Lemma \ref{Bramble}
	\begin{equation}
		\begin{split}
			\vert \bar{R}_\Gamma(E_{D,\Gamma}\mathbf{v}_\Gamma)\vert^2_{1/2,\Gamma} \leq C\bigg(1+ \log\left(\frac{Hk^2}{h}\right)\bigg)^2\vert \bar{R}_\Gamma \mathbf{v}_\Gamma \vert ^2_{1/2,\Gamma}\\
			\leq C\bigg(1+ \log\left(\frac{Hk^2}{h}\right)\bigg)^2\vert \bar{R}_\Gamma \mathbf{v}_\Gamma \vert ^2_{S_\Gamma} \leq C\bigg(1+ \log\left(\frac{Hk^2}{h}\right)\bigg)^2\Vert \mathbf{v}_\Gamma \Vert ^2_{\widetilde{S}_\Gamma}\textit{.}
		\end{split}
	\end{equation}
	Consequently we have 
	\begin{equation}
		\vert E_D \mathbf{v} \vert^2_{\widetilde{S}} \leq C\frac{1}{\tilde{\beta}^2}\bigg(1+ \log\left(\frac{Hk^2}{h}\right)\bigg)^2\Vert \mathbf{v}_\Gamma \Vert ^2_{\widetilde{S}_\Gamma} = C \frac{1}{\tilde{\beta}^2}\bigg(1+ \log\left(\frac{Hk^2}{h}\right)\bigg)^2\vert \mathbf{v} \vert ^2_{\widetilde{S}}\textit{.}
	\end{equation}
\end{proof}

We have the following lemma (proof in \cite{li2006bddc}):

\begin{lemma}\label{lemma6}
	Any vector of the form $\mathbf{u} = (\mathbf{0},p_0)\in \widehat{\mathbf{V}}_{\Gamma,B}\times Q_0$ is an eigenvector of the preconditioner operator $M^{-1}\widehat{S}$ with eigenvalue equal to 1.
\end{lemma}
\begin{theorem}
	Let Assumptions 1 and 2 hold. The preconditioned operator $M^{-1}\widehat{S}$ is then symmetric, positive definite
	with respect to the bilinear form $\langle\cdot,\cdot\rangle_{\widehat{S}}$ on the benign space $\widehat{\mathbf{V}}_{\Gamma,B}\times Q_0$. Its minimum eigenvalue is 1 and its maximum eigenvalue is bounded by
	\begin{align}
		C \frac{1}{\tilde{\beta}^2}\bigg(1+\log\left(\frac{Hk^2}{h}\right)\bigg)^2\text{.}
	\end{align}
	Here, $C$ is a constant which is independent of $H$, $h$ and the number of subdomains, and $\tilde{\beta}$ is the inf-sup stability constant defined in \eqref{infsupdisc}.
\end{theorem}
\begin{proof}
	We know from Lemma \ref{lemma6}, that any vector of the form $\mathbf{u} = (\mathbf{0},p_0)\in \widehat{\mathbf{V}}_{\Gamma,B}\times Q_0$ is an eigenvector of the preconditioned operator $M^{-1}\widehat{S}$ with an eigenvalue equal to 1. It is sufficient to find lower and upper bounds of the quotient $\big \langle M^{-1}\widehat{S} \mathbf{u},\mathbf{u}\big \rangle_{\widehat{S}} / \langle\mathbf{u},\mathbf{u}\rangle_{\widehat{S}}$, for any $\mathbf{u} = (\mathbf{u}_\Gamma,p_0)\in \widehat{\mathbf{V}}_{\Gamma,B}\times Q_0$, where $\mathbf{u}_\Gamma$ is non zero and therefore $\langle\mathbf{u},\mathbf{u}\rangle_{\widehat{S}}> 0$.\\
	\textit{Lower bound}: Given  $\mathbf{u}\in \widehat{\mathbf{V}}_{\Gamma,B}\times Q_0$, let
	\begin{equation}\label{25}
		\mathbf{v} = \widetilde{S}^{-1} \widetilde{R}_D \widetilde{S} \mathbf{u} \in \widetilde{\mathbf{V}}_{\Gamma,B}\times Q_0 \text{.}
	\end{equation}
	We have from the fact that $\widetilde{R}^T \widetilde{R}_D = \widetilde{R}^T_D \widetilde{R} = I$, 
	\begin{equation}\label{26}
		\langle\mathbf{u},\mathbf{u}\rangle_{\widehat{S}} = \mathbf{u}^T \widehat{S} \widetilde{R}^T_D \widetilde{R} \mathbf{u} = \mathbf{u}^T \widehat{S} \widetilde{R}^T_D \widetilde{S}^{-1} \widetilde{S} \widetilde{R} \mathbf{u} = \langle\mathbf{v},\widetilde{R}\mathbf{u}\rangle_{\widetilde{S}}.
	\end{equation}
	From the Cauchy-Schwartz inequality and the fact that $\widehat{S} = \widetilde{R}^T \widetilde{S} \widetilde{R}$, we find that
	\begin{equation}\label{27}
		\langle\mathbf{v},\widetilde{R}\mathbf{u}\rangle_{\widetilde{S}} \leq \langle\mathbf{v},\mathbf{v}\rangle^{1/2}_{\widetilde{S}}\langle\widetilde{R}\mathbf{u},\widetilde{R}\mathbf{u}\rangle^{1/2}_{\widetilde{S}} = \langle\mathbf{v},\mathbf{v}\rangle^{1/2}_{\widetilde{S}}\langle\mathbf{u},\mathbf{u}\rangle^{1/2}_{\widetilde{S}}\text{.}
	\end{equation}
	Therefore from \eqref{26} and \eqref{27},
	\begin{equation}\label{28}
		\langle\mathbf{u},\mathbf{u}\rangle_{\widetilde{S}} \leq \langle\mathbf{v},\mathbf{v}\rangle_{\widetilde{S}} \text{.}
	\end{equation}
	Since,
	\begin{equation}\label{29}
		\langle\mathbf{v},\mathbf{v}\rangle_{\widetilde{S}}=\mathbf{u}^T \widehat{S} \widetilde{R}^T_D \widetilde{S}^{-1} \widetilde{S} \widetilde{S}^{-1} \widetilde{R}_D \widehat{S} \mathbf{u} = \big \langle \mathbf{u},\widetilde{R}^T_D \widetilde{S}^{-1} \widetilde{R}_D \widehat{S} \mathbf{u}\rangle_{\widehat{S}} = \big \langle \mathbf{u}, M^{-1}\widehat{S}\mathbf{u}\big \rangle_{\widehat{S}}\text{,}
	\end{equation}
	we obtain, from equations \eqref{28} and \eqref{29}, that $\langle\mathbf{u},\mathbf{u}\rangle_{\widetilde{S}} \leq \big \langle \mathbf{u}, M^{-1}\widehat{S}\mathbf{u}\big \rangle_{\widehat{S}}$, which gives a lower bound of 1 for the eigenvalues. Then from Lemma \ref{lemma6}, we know that 1 is the minimum eigenvalue of the preconditioned operator.\\
	\textit{Upper bound}:  Given  $\mathbf{u}\in \widehat{\mathbf{V}}_{\Gamma,B}\times Q_0$, take $\mathbf{v}\in \widetilde{\mathbf{V}}_{\Gamma,B}\times Q_0$ as in \eqref{25}.
	We have, $\widetilde{R}^T_D \mathbf{v} = M^{-1}\widehat{S} \mathbf{u}$. Since $\widehat{S} = \widetilde{R}^T \widetilde{S} \widetilde{R}$ and by using Lemma \ref{lemma5}, we have
	\begin{equation}
		\begin{split}
			\langle M^{-1}\widehat{S}\mathbf{u}, M^{-1}\widehat{S} \mathbf{u}\rangle_{\widehat{S}} = \langle \widetilde{R}^T_D\mathbf{v}, \widetilde{R}^T_D\mathbf{v} \rangle_{\widehat{S}} = \langle \widetilde{R} \widetilde{R}^T_D\mathbf{v}, \widetilde{R} \widetilde{R}^T_D\mathbf{v} \rangle_{\widetilde{S}} \\
			= 	|E_D\mathbf{v}|^2_{\widetilde{S}} \leq C^2 \frac{1}{\tilde{\beta}^2}\bigg(1+\log\left(\frac{Hk^2}{h}\right)\bigg)^2|\mathbf{v}|^2_{\widetilde{S}}\text{.}
		\end{split}
	\end{equation}
	Therefore from equation \eqref{29}, we have
	\begin{equation}\label{30}
		\langle M^{-1}\widehat{S}\mathbf{u}, M^{-1}\widehat{S} \mathbf{u}\rangle_{\widehat{S}} \leq C^2 \frac{1}{\tilde{\beta}^2}\bigg(1+\log\left(\frac{Hk^2}{h}\right)\bigg)^2 		\langle \mathbf{u}, M^{-1}\widehat{S} \mathbf{u}\rangle_{\widehat{S}}\text{.}
	\end{equation}
	Using the Cauchy-Schwarz inequality and equation \eqref{30}, we have 
	\begin{equation}
		\begin{split}
			\langle \mathbf{u}, M^{-1}\widehat{S} \mathbf{u}\rangle_{\widehat{S}} \leq \langle M^{-1}\widehat{S}\mathbf{u}, M^{-1}\widehat{S} \mathbf{u}\rangle^{1/2}_{\widehat{S}}\langle \mathbf{u}, \mathbf{u}\rangle^{1/2}_{\widehat{S}} \\
			\leq C \frac{1}{\tilde{\beta}}\bigg(1+\log\left(\frac{Hk^2}{h}\right) \bigg) \langle \mathbf{u},  \mathbf{u} \rangle^{1/2}_{\widehat{S}} \langle \mathbf{u}, M^{-1}\widehat{S} \mathbf{u}\rangle^{1/2}_{\widehat{S}} \text{.}
		\end{split}
	\end{equation}
	This gives
	\begin{equation}
		\langle \mathbf{u}, M^{-1}\widehat{S} \mathbf{u}\rangle_{\widehat{S}} \leq C \frac{1}{\tilde{\beta}^2}\bigg(1+\log\left(\frac{Hk^2}{h}\right) \bigg)^2 \langle \mathbf{u},  \mathbf{u} \rangle_{\widehat{S}} \text{,}
	\end{equation}
	and the upper bound of the theorem.
\end{proof}

\section{Satisfying the Assumptions}
\label{sec:9}
To satisfy the assumptions 1 and 2 in of the previous section we have to choose properly the primal constraints for the interface velocity space. In particular to satisfy the Assumption 1, it is not sufficient to choose as primal constraints the subdomain vertices, i.e. make both components of the velocity continuous at those nodes, but some extra edge constraints are necessary. To do so, for each interface edge $\Gamma_{ij}$, which is shared by a pair of subdomains $\Omega_i$ and $\Omega_j$, we impose
\begin{align}
	\int_{\Gamma_ij} \mathbf{v}_\Gamma^{(i)}\cdot\mathbf{n}_{ij} = \int_{\Gamma_ij} \mathbf{v}_\Gamma^{(j)}\cdot\mathbf{n}_{ij}
\end{align}
for a fixed selection of the normal $\mathbf{n}_{ij}$ of $\Gamma_{i,h}$.
Proceeding the discussion with this first choice, after changing the variables, the dual interface velocity component will vanish at the subdomain vertices and its normal component will have a weighted zero average over each $\Gamma_{ij}$, i.e. 
\begin{align*}
\int_{\Gamma_i} \mathbf{v_\Gamma}^{(i)}\cdot\mathbf{n}_{ij} = \int_{\Gamma_j} \mathbf{v_\Gamma}^{(j)}\cdot\mathbf{n}_{ij} = 0.
\end{align*}
By the definition of the average operator $E_{D,\Delta}$ we have that the average interface velocity is $E_{D,\Delta}\mathbf{v}_\Delta = \frac{1}{2}(\mathbf{v}_\Delta^{(i)}+\mathbf{v}_\Delta^{(j)})$ on each edge and hence 
\begin{align}
	\int_{\Gamma_ij} (E_{D,\Delta}\mathbf{v}_\Delta)^{(i)}\cdot\mathbf{n}_{ij} = 0\text{.}
\end{align}
In our codes we also choose a strong condition, we decide to require that the integral of both velocity components have common values across each interface edge
\begin{align}
	\int_{\Gamma_ij} \mathbf{v}_\Gamma^{(i)}\big|_x = \int_{\Gamma_j} \mathbf{v}_\Gamma^{(j)}\big|_x\text{,}\quad
	\int_{\Gamma_ij} \mathbf{v}_\Gamma^{(i)}\big|_y = \int_{\Gamma_j} \mathbf{v}_\Gamma^{(j)}\big|_y\text{.}
\end{align} 
In this way, we clearly satisfy Assumption \ref{ass1}. The advantage of this condition is that it is easiest to implement and, enlarging a little the coarse space, it yields a faster convergence.
Assumption 2 is also satisfied, requiring only vertices as primal constraints, and it derives directly from the following lemma, proved in \cite{bertoluzza2017bddc}:
\begin{lemma}\label{logest}
	For all $\mathbf{v}_\Gamma \in \widetilde{\mathbf{V}}_\Gamma$ we have:
	\begin{align}
		|E_{D,\Gamma}\mathbf{v}_\Gamma|_{1/2,\Gamma} \leq C \left(1+\log\left(\frac{Hk^2}{h}\right)\right)|\mathbf{v}_\Gamma|_{1/2,\Gamma}
	\end{align}
	with $C$ positive constant independent of $H$, $h$ and the number of subdomains.
\end{lemma}

\section{Numerical Results}
\label{sec:10}
In this section, we provide some numerical tests to study the behavior of the BDDC preconditioner with respect to the mesh size $h$,
the number of subdomains $N$ and the shape of the polygonal mesh elements. We solve the Stokes equations on the unit square domain $\Omega = [0,1]\times[0,1]$, applying homogeneous Dirichlet boundary conditions on the whole $\partial\Omega$. 
We choose the load term $\mathbf{f}$ by imposing that the analytical solution is
\begin{align*}
	\mathbf{u}(x,y) = 
	\left(
	\begin{array}{c}
		-\sin(\pi x)\sin(\pi x) \sin(2\pi y)\\
		\sin(\pi y)\sin(\pi y)\sin(2\pi x)
	\end{array}
	\right),\quad p(x,y) = \sin(\pi x)-\sin(\pi y)\text{.}
\end{align*}

\begin{figure*}[!h]	
	\centering
	\subfigure[QUAD mesh.]{
		\includegraphics[width=5.5cm]{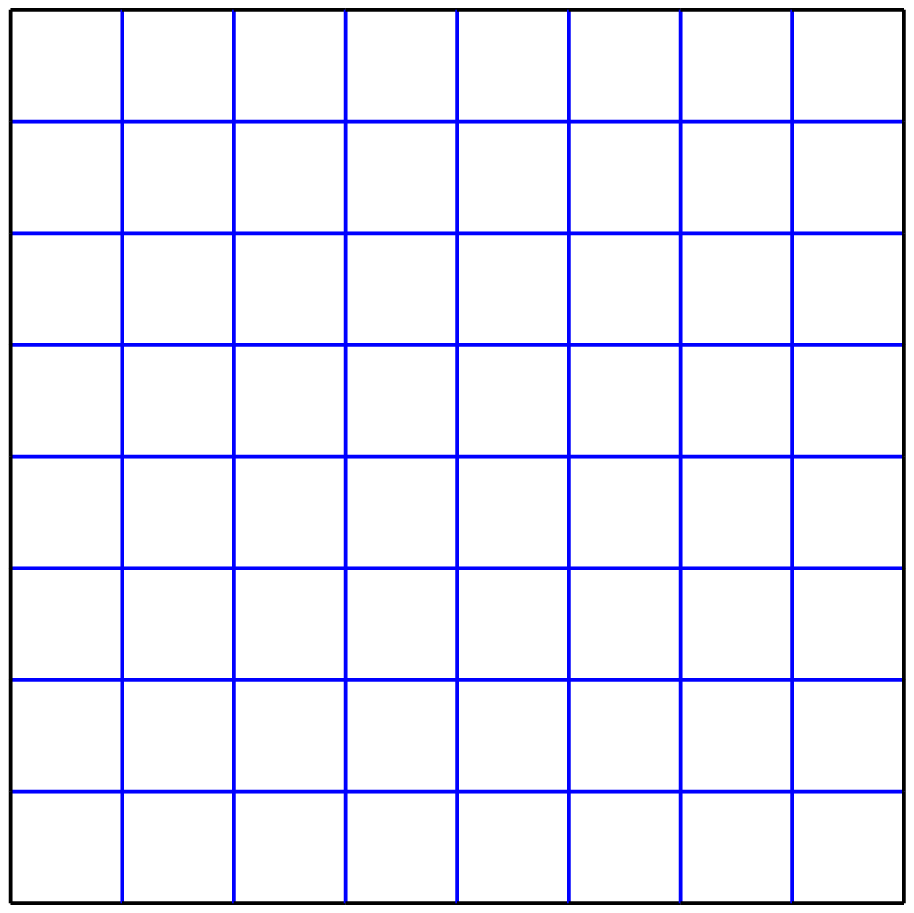}
	}
	\subfigure[HEXA mesh.]{
		\includegraphics[width=5.5cm]{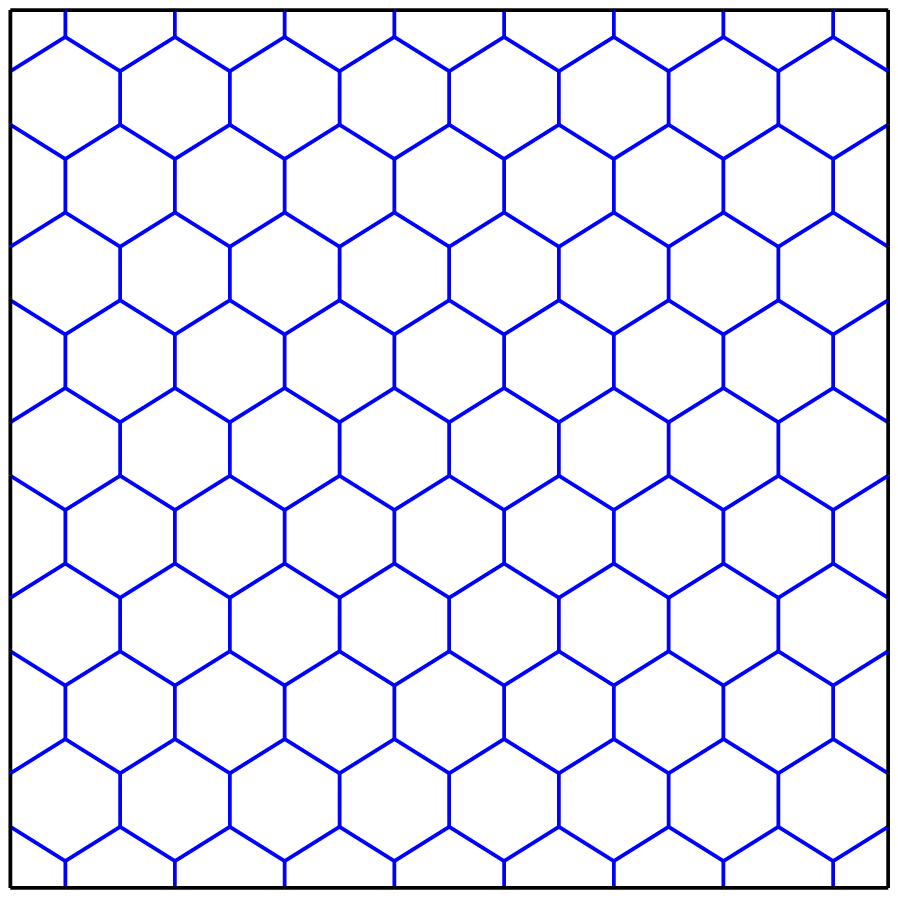}
	}
	\\
	\subfigure[TRI mesh.]{
		\includegraphics[width=5.5cm]{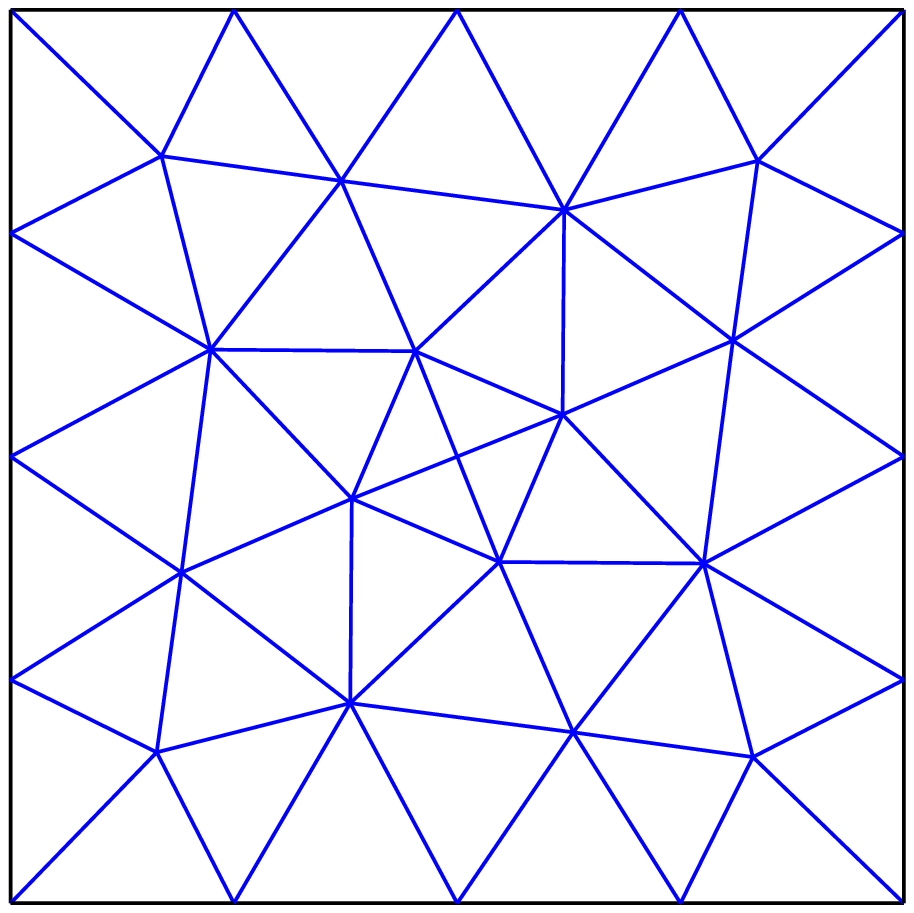}
	}
	\subfigure[CVT mesh.]{
		\includegraphics[width=5.5cm]{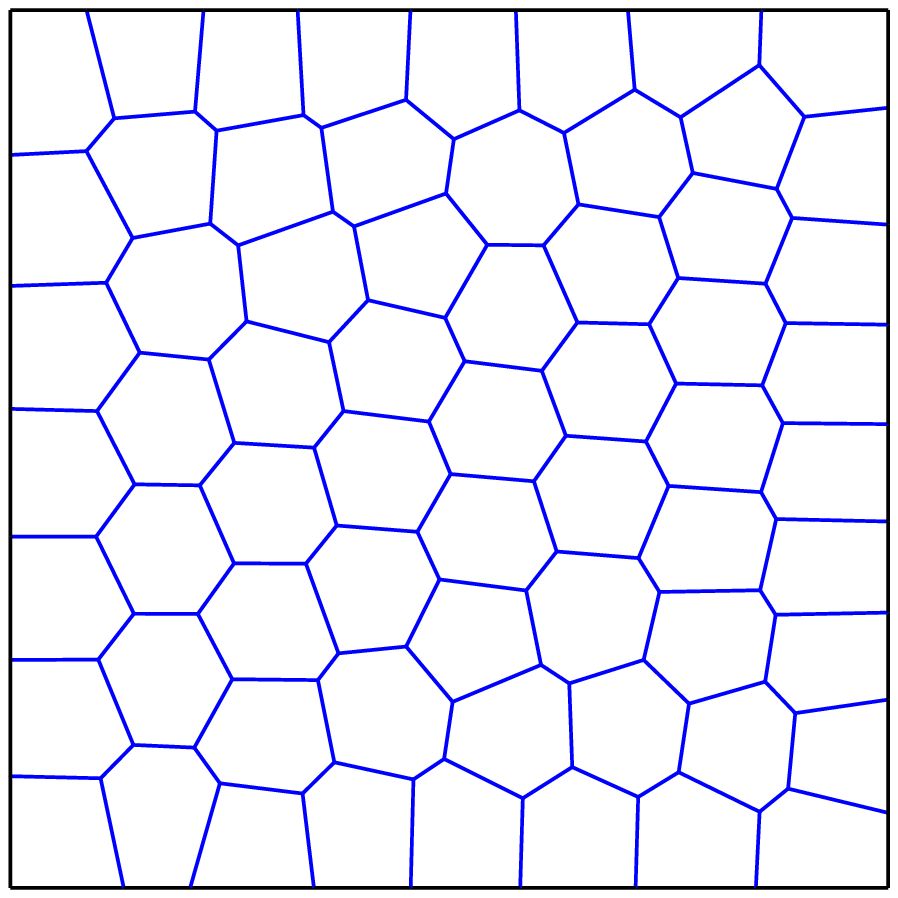}
	}
	\caption{Examples of the different type of meshes we consider in our numerical experiments.}
	\label{meshfig}
\end{figure*}

\begin{table}[htbp]
		\caption{GMRES iteration counts to solve the interface saddle-point problem  without preconditioner varying the mesh size $h$ and the number of subdomains $N=1/H^2$.}\label{GMRES}
	\subtable[QUAD meshes.\label{GMRESqua}]{%
		\scalebox{0.73}{
			\begin{tabular}{|c|c c c c c|}
				\hline
				\backslashbox{1/H}{1/h} & \minitab[c]{8 \\ it} & \minitab[c]{16 \\ it} & \minitab[c]{32 \\ it} & \minitab[c]{64 \\ it} & \minitab[c]{128 \\ it}\\
				\hline
				2 & 7 & 13 & 18 & 27 & 40 \\
				\hline
				4 & x & 40 & 57 & 79 & 99 \\
				\hline
				8 & x & x & 140 & 203 & 282 \\
				\hline
				16 & x & x & x & 295 & 437 \\
				\hline
				32 & x & x & x & x & 605 \\
				\hline
			\end{tabular}
		}
	}
	\subtable[HEXA meshes.\label{GMRESexa}]{%
		\scalebox{0.73}{
			\begin{tabular}{|c|c c c c c|}
				\hline
				\backslashbox{1/H}{1/h} & \minitab[c]{8 \\ it} & \minitab[c]{16 \\ it} & \minitab[c]{32 \\ it} & \minitab[c]{64 \\ it} & \minitab[c]{128 \\ it}\\
				\hline
				2 & 18 & 23 & 31 & 42 & 56\\
				\hline
				4 & x & 91 & 131 & 163 & 199\\
				\hline
				8 & x & x & 221 & 295 & 434\\
				\hline
				16 & x & x & x & 486 & 583\\
				\hline
				16 & x & x & x & x & 976\\
				\hline
			\end{tabular}
		}
	}\\
	\subtable[TRI meshes.\label{GMREStri}]{%
		\scalebox{0.73}{
			\begin{tabular}{|c|c c c c c|}
				\hline
				\backslashbox{1/H}{1/h} & \minitab[c]{8 \\ it} & \minitab[c]{16 \\ it} & \minitab[c]{32 \\ it} & \minitab[c]{64 \\ it} & \minitab[c]{128 \\ it}\\
				\hline
				2 & 12 & 17 & 25 & 38 & 52 \\
				\hline
				4 & x & 42 & 63 & 82 & 105 \\
				\hline
				8 & x & x & 127 & 190 & 286 \\
				\hline
				16 & x & x & x & 380 & 569 \\
				\hline
				32 & x & x & x & x & 1130 \\
				\hline
			\end{tabular}
		}
	}
	\subtable[CVT meshes.\label{GMRESvor}]{%
		\scalebox{0.73}{
			\begin{tabular}{|c|c c c c c|}
				\hline
				\backslashbox{1/H}{1/h} & \minitab[c]{8 \\ it} & \minitab[c]{16 \\ it} & \minitab[c]{32 \\ it} & \minitab[c]{64 \\ it} & \minitab[c]{128 \\ it}\\
				\hline
				2 & 31 & 44 & 55 & 79 & 101 \\
				\hline
				4 & x & 89 & 104 & 134 & 172 \\
				\hline
				8 & x & x & 267 & 307 & 336 \\
				\hline
				16 & x & x & x & 885 & 924 \\
				\hline
				32 & x & x & x & x & 2522 \\
				\hline
			\end{tabular}
		}
	}	
\end{table}

\begin{table}[htbp]
		\caption{PCG iteration counts to solve the interface saddle-point problem with BDDC preconditioner, varying the mesh size $h$ and the number of subdomains $N=1/H^2$. The primal coarse space is spanned only by the subdomain vertices.}\label{BDDCV}
	\subtable[QUAD meshes.\label{BDDCVqua}]{%
		\scalebox{0.73}{
			\begin{tabular}{|c|c c c c c|}
				\hline
				\backslashbox{1/H}{1/h} & \minitab[c]{8 \\ it} & \minitab[c]{16 \\ it} & \minitab[c]{32 \\ it} & \minitab[c]{64 \\ it} & \minitab[c]{128 \\ it}\\
				\hline
				2 & 7 & 7 & 7 & 7 & 8 \\
				\hline
				4 & x & 12 & 13 & 15 & 17 \\
				\hline
				8 & x & x & 22 & 23 & 29 \\
				\hline
				16 & x & x & x & 22 & 26 \\
				\hline
				32 & x & x & x & x & 22 \\
				\hline
			\end{tabular}
		}
	}
	\subtable[HEXA meshes.\label{BDDCVexa}]{%
		\scalebox{0.73}{
			\begin{tabular}{|c|c c c c c|}
				\hline
				\backslashbox{1/H}{1/h} & \minitab[c]{8 \\ it} & \minitab[c]{16 \\ it} & \minitab[c]{32 \\ it} & \minitab[c]{64 \\ it} & \minitab[c]{128 \\ it}\\
				\hline
				2 & 8 & 9 & 9 & 9 & 9\\
				\hline
				4 & x & 18 & 19 & 21 & 22\\
				\hline
				8 & x & x & 26 & 30 & 33\\
				\hline
				16 & x & x & x & 28 & 30\\
				\hline
				32 & x & x & x & x & 29\\
				\hline
			\end{tabular}
		}
	}\\
	\subtable[TRI meshes.\label{BDDCVtri}]{%
		\scalebox{0.73}{
			\begin{tabular}{|c|c c c c c|}
				\hline
				\backslashbox{1/H}{1/h} & \minitab[c]{8 \\ it} & \minitab[c]{16 \\ it} & \minitab[c]{32 \\ it} & \minitab[c]{64 \\ it} & \minitab[c]{128 \\ it}\\
				\hline
				2 & 8 & 8 & 8 & 9 & 9 \\
				\hline
				4 & x & 15 & 17 & 19 & 20 \\
				\hline
				8 & x & x & 18 & 24 & 27 \\
				\hline
				16 & x & x & x & 20 & 25 \\
				\hline
				32 & x & x & x & x & 20 \\
				\hline
			\end{tabular}
		}
	}
	\subtable[CVT meshes.\label{BDDCVvor}]{%
		\scalebox{0.73}{
			\begin{tabular}{|c|c c c c c|}
				\hline
				\backslashbox{1/H}{1/h} & \minitab[c]{8 \\ it} & \minitab[c]{16 \\ it} & \minitab[c]{32 \\ it} & \minitab[c]{64 \\ it} & \minitab[c]{128 \\ it}\\
				\hline
				2 & 16 & 17 & 17 & 17 & 17 \\
				\hline
				4 & x & 26 & 30 & 32 & 33 \\
				\hline
				8 & x & x & 36 & 41 & 41 \\
				\hline
				16 & x & x & x & 50 & 50 \\
				\hline
				32 & x & x & x & x & 51 \\
				\hline
			\end{tabular}
		}
	}	
\end{table}

\begin{table}[htbp]
		\caption{PCG iteration counts to solve the interface problem with BDDC preconditioner, varying the mesh size $h$ and the number of subdomains $N=1/H^2$. The primal coarse space is spanned by the subdomain vertices and only one basis function per subdomain edge.}\label{BDDCVN}
	\centering
	\subtable[QUAD meshes.\label{BDDCVNqua}]{%
		\scalebox{0.75}{
			\begin{tabular}{|c|c c| c c | c c | c c| c c |}
				\hline
				\backslashbox{1/H}{1/h} & \minitab[c]{ \\ $\kappa_2$} & \minitab[c]{8 \\ it} & \minitab[c]{ \\ $\kappa_2$} & \minitab[c]{16 \\ it} & \minitab[c]{ \\ $\kappa_2$} & \minitab[c]{32 \\ it} & \minitab[c]{ \\ $\kappa_2$} & \minitab[c]{64 \\ it} & \minitab[c]{ \\ $\kappa_2$} & \minitab[c]{128 \\ it} \\
				\hline
				2 & 1,82 & 7 & 2,11 & 7	& 2,37 & 7 & 2,80 & 7 & 3,24 & 8 \\
				\hline
				4 & x & x & 4,40 & 9 & 5,75 & 10 & 7,20 & 11 & 8,76 & 12 \\
				\hline
				8 & x & x & x & x & 5,78 & 13 & 7,81 & 15 & 10,09 & 16\\
				\hline
				16 & x & x & x & x	& x & x & 6,16 & 16 & 8,46 & 19 \\
				\hline
				32 & x & x & x & x & x & x & x & x & 6,29 & 16 \\
				\hline
			\end{tabular}
		}
	}
	\subtable[HEXA meshes.\label{BDDCNEexa}]{%
		\scalebox{0.75}{
			\begin{tabular}{|c|c c| c c | c c | c c| c c |}
				\hline
				\backslashbox{1/H}{1/h} & \minitab[c]{ \\ $\kappa_2$} & \minitab[c]{8 \\ it} & \minitab[c]{ \\ $\kappa_2$} & \minitab[c]{16 \\ it} & \minitab[c]{ \\ $\kappa_2$} & \minitab[c]{32 \\ it} & \minitab[c]{ \\ $\kappa_2$} & \minitab[c]{64 \\ it} & \minitab[c]{ \\ $\kappa_2$} & \minitab[c]{128 \\ it}\\
				\hline
				2 & 3,35 & 9 & 4,45 & 9 & 5,49 & 9 & 6,64 & 9 & 6,95 & 9 \\
				\hline
				4 & x & x & 5,32 & 13 & 6,86 & 14 & 8,45 & 15 & 10,17 & 16 \\
				\hline
				8 & x & x & x & x & 7,34 & 17 & 10,12 & 19 & 12,37 & 20\\
				\hline
				16 & x & x & x & x	& x & x & 7,97 & 18 & 11,07 & 22 \\
				\hline
				32 & x & x & x & x & x & x & x & x & 8,20 & 18 \\
				\hline
			\end{tabular}
		}
	}
	\subtable[TRI meshes.\label{BDDCNEtri}]{%
		\scalebox{0.75}{
			\begin{tabular}{|c|c c| c c | c c | c c| c c |}
				\hline
				\backslashbox{1/H}{1/h} & \minitab[c]{ \\ $\kappa_2$} & \minitab[c]{8 \\ it} & \minitab[c]{ \\ $\kappa_2$} & \minitab[c]{16 \\ it} & \minitab[c]{ \\ $\kappa_2$} & \minitab[c]{32 \\ it} & \minitab[c]{ \\ $\kappa_2$} & \minitab[c]{64 \\ it} & \minitab[c]{ \\ $\kappa_2$} & \minitab[c]{128 \\ it} \\
				\hline
				2 & 2,73 & 8 & 3,51 & 8 & 4,28 & 8 & 5,12 & 9 & 6,32 & 9 \\
				\hline
				4 & x & x & 4,01 & 11 & 5,20 & 12 & 6,54 & 13 & 7,98 & 14 \\
				\hline
				8 & x & x & x & x & 5,01 & 15 & 6,83 & 16 & 8,93 & 18\\
				\hline
				16 & x & x & x & x	& x & x & 5,25 & 15 & 7,28 & 17 \\
				\hline
				32 & x & x & x & x & x & x & x & x & 5,32 & 15 \\
				\hline
			\end{tabular}
		}
	}
	\subtable[CVT meshes.\label{BDDCVNvor}]{%
		\scalebox{0.75}{
			\begin{tabular}{|c|c c| c c | c c | c c| c c |}
				\hline
				\backslashbox{1/H}{1/h} & \minitab[c]{ \\ $\kappa_2$} & \minitab[c]{8 \\ it} & \minitab[c]{ \\ $\kappa_2$} & \minitab[c]{16 \\ it} & \minitab[c]{ \\ $\kappa_2$} & \minitab[c]{32 \\ it} & \minitab[c]{ \\ $\kappa_2$} & \minitab[c]{64 \\ it} & \minitab[c]{ \\ $\kappa_2$} & \minitab[c]{128 \\ it} \\
				\hline
				2 & 5,82 & 14 & 6,97 & 15 & 8,16 & 16 & 9,32 & 16 & 10,23 & 16 \\
				\hline
				4 & x & x & 10,20 & 20 & 13,87 & 21 & 15,98 & 22 & 17,22 & 23 \\
				\hline
				8 & x & x & x & x & 22,24 & 27 & 21,43 & 28 & 23,13 & 28\\
				\hline
				16 & x & x & x & x	& x & x & 30,12 & 34 & 28,34 & 33 \\
				\hline
				32 & x & x & x & x & x & x & x & x & 30,28 & 35 \\
				\hline
			\end{tabular}
		}
	}	
\end{table}

\begin{table}[htbp]
		\caption{PCG iteration counts to solve the interface saddle-point problem with BDDC preconditioner, varying the mesh size $h$ and the number of subdomains $N=1/H^2$. The primal coarse space is spanned by the subdomain vertices and two basis functions per subdomain edge.}\label{BDDCVE}
	\centering
	\subtable[QUAD meshes.\label{BDDCVEqua}]{%
		\scalebox{0.75}{
			\begin{tabular}{|c|c c| c c | c c | c c| c c |}
				\hline
				\backslashbox{1/H}{1/h} & \minitab[c]{ \\ $\kappa_2$} & \minitab[c]{8 \\ it} & \minitab[c]{ \\ $\kappa_2$} & \minitab[c]{16 \\ it} & \minitab[c]{ \\ $\kappa_2$} & \minitab[c]{32 \\ it} & \minitab[c]{ \\ $\kappa_2$} & \minitab[c]{64 \\ it} & \minitab[c]{ \\ $\kappa_2$} & \minitab[c]{128 \\ it} \\
				\hline
				2 & 1,48 & 7 & 1,68 & 7	& 1,90 & 7 & 2,18 & 8 & 2,51 & 7 \\
				\hline
				4 & x & x & 2,80 & 9 & 3,73 & 10 & 4,78 & 10 & 5,93 & 11 \\
				\hline
				8 & x & x & x & x & 2,99 & 10 & 4,05 & 11 & 5,20 & 13\\
				\hline
				16 & x & x & x & x	& x & x & 2,72 & 9 & 3,67 & 10 \\
				\hline
				32 & x & x & x & x & x & x & x & x & 2,64 & 8 \\
				\hline
			\end{tabular}
		}
	} \quad
	\subtable[HEXA meshes.\label{BDDCVEexa}]{%
		\scalebox{0.75}{
			\begin{tabular}{|c|c c| c c | c c | c c| c c |}
				\hline
				\backslashbox{1/H}{1/h} & \minitab[c]{ \\ $\kappa_2$} & \minitab[c]{8 \\ it} & \minitab[c]{ \\ $\kappa_2$} & \minitab[c]{16 \\ it} & \minitab[c]{ \\ $\kappa_2$} & \minitab[c]{32 \\ it} & \minitab[c]{ \\ $\kappa_2$} & \minitab[c]{64 \\ it} & \minitab[c]{ \\ $\kappa_2$} & \minitab[c]{128 \\ it}\\
				\hline
				2 & 3,33 & 9 & 4,29 & 	10	& 5,36 & 10 & 6,68 & 10 & 8.08 & 11\\
				\hline
				4 & x & x & 4,21 & 12 & 5,29 & 13 & 6,58 & 15 & 7,90 & 15\\
				\hline
				8 & x & x & x & x & 4,59 & 13 & 5,90 & 14 & 6,65 & 15\\
				\hline
				16 & x & x & x & x & x & x & 4,79 & 14 & 5,12 & 13\\
				\hline
				32 & x & x & x & x & x & x & x & x & 4,35 & 12\\
				\hline
			\end{tabular}
		}
	} \quad
	\subtable[TRI meshes.\label{BDDCVEtri}]{%
		\scalebox{0.75}{
			\begin{tabular}{|c|c c| c c | c c | c c| c c |}
				\hline
				\backslashbox{1/H}{1/h} & \minitab[c]{ \\ $\kappa_2$} & \minitab[c]{8 \\ it} & \minitab[c]{ \\ $\kappa_2$} & \minitab[c]{16 \\ it} & \minitab[c]{ \\ $\kappa_2$} & \minitab[c]{32 \\ it} & \minitab[c]{ \\ $\kappa_2$} & \minitab[c]{64 \\ it} & \minitab[c]{ \\ $\kappa_2$} & \minitab[c]{128 \\ it} \\
				\hline
				2 & 2,49 & 9 & 3,41 & 9 & 4,38 & 10 & 5,40 & 10 & 6,55 & 10 \\
				\hline
				4 & x & x & 2,96 & 10 & 3,85 & 11 & 5,17 & 12 & 6,36 & 14 \\
				\hline
				8 & x & x & x & x & 3,26 & 10 & 4,26 & 12 & 5,33 & 13\\
				\hline
				16 & x & x & x & x	& x & x & 3,44 & 9 & 4,42 & 11 \\
				\hline
				32 & x & x & x & x & x & x & x & x & 3,49 & 8 \\
				\hline
			\end{tabular}
		}
	}\quad
	\subtable[CVT meshes.\label{BDDCVEvor}]{%
		\scalebox{0.75}{
			\begin{tabular}{|c|c c| c c | c c | c c| c c |}
				\hline
				\backslashbox{1/H}{1/h} & \minitab[c]{ \\ $\kappa_2$} & \minitab[c]{8 \\ it} & \minitab[c]{ \\ $\kappa_2$} & \minitab[c]{16 \\ it} & \minitab[c]{ \\ $\kappa_2$} & \minitab[c]{32 \\ it} & \minitab[c]{ \\ $\kappa_2$} & \minitab[c]{64 \\ it} & \minitab[c]{ \\ $\kappa_2$} & \minitab[c]{128 \\ it} \\
				\hline
				2 & 4,35 & 13 & 5,27 & 14	& 6,63 & 15 & 7,31 & 16 & 8,28 & 16 \\
				\hline
				4 & x & x & 5,20 & 15 & 10,22 & 18 & 13,00 & 19 & 15,63 & 20 \\
				\hline
				8 & x & x & x & x & 9,03 & 20 & 17,52 & 23 & 18,41 & 22\\
				\hline
				16 & x & x & x & x	& x & x & 12,29 & 21 & 19,21 & 24 \\
				\hline
				32 & x & x & x & x & x & x & x & x & 14,43 & 23 \\
				\hline
			\end{tabular}
		}
	}	
\end{table}

In the following tables, we report the number of iterations to solve the global interface saddle-point problem (\ref{globInt}) with the non-preconditioned GMRES method or the PCG method, accelerated by BDDC.
Where possible, we estimate the extreme eigenvalues using the Lanczos trick. Both in case of PCG and GMRES, we set the tolerance for the relative residual error to $10^{-6}$. Note that in the tables we marked with an "x" the numerical tests that we do not have performed because they are not significant.\\
Our tests have been executed on different types of polygonal meshes and using the VEM discretization with degree $k=2$ with the divergence free approach, that means having polynomials of degree 2 on the boundary of each element for the velocity and piecewise constant functions for the pressure. We underline the fact that we would have obtained the same behavior, both in terms of number of iterations and spectral condition number number, also in the case of neglecting the divergence free property, because the interface problem and the preconditioner are exactly the same due to the decomposition technique used in (\ref{discVQ}) and (\ref{discViQi}).

The polygonal meshes considered are quadrilateral (QUAD), hexagonal (HEXA), triangular (TRI) and Voronoi (CVT) (Figure \ref{meshfig}).

Table \ref{GMRES} reports the number of iterations to solve the interface saddle-point problem with the non-preconditioned GMRES. 
As expected, we observe that the iteration counts grow when the number of subdomains increases and the mesh size decreases. 

Table \ref{BDDCV} reports the number of iterations to solve the interface saddle-point problem with
PCG, preconditioned by BDDC, considering as primal constraints only the subdomain vertices.
In this case the solver appears to be scalable, since, moving along the diagonals of the table, the iterations remain bounded when the number of subdomains increase,
and quasi-optimal, since, moving along the rows of the table, the growth of iterations seems logarithmic.
The results also show that the solver suffers more on the Voronoi meshes than on the others. We recall that with this choice of primal constraints the assumption \ref{ass1} is not satisfied, therefore the preconditioned system is not positive definite and we are not able to give an estimate on the eigenvalues.

Table \ref{BDDCVN} reports the spectral condition number of the preconditioned system and the iteration counts to solve the interface problem with the PCG method, preconditioned by BDDC, where the primal constraints are the subdomain vertices and one basis function for each subdomain edge.
In this case both the assumptions are satisfied, therefore the system is symmetric and positive definite and we are able to give an estimate of the eigenvalues. 
The results confirm the theoretical estimates, since both the condition number and the number of iterations are independent of number of subdomains (scalability) and exhibit a logarithmic growth with respect to the ratio $H/h$ (quasi-optimality). 

Table \ref{BDDCVE} reports the spectral condition number of the preconditioned system and the iteration counts to solve the interface problem with the PCG method, preconditioned by BDDC, where the primal constraints are the subdomain vertices and two basis functions for each subdomain edge.
In this case the system is again symmetric and positive definite, thus we are able to give an estimate of the eigenvalues. 
Both the condition number and the iteration counts exhibit a scalable and quasi-optimal behavior as before, but in this case the convergence is faster since the coarse problem is slightly larger.

We recall that our code is implemented in Matlab and the tests were performed in serial, therefore we 
do not provide an analysis on the time of computations.
\begin{figure}[!t]
	\centering
	\subfigure[The primal space consists of the subdomain vertices and two basis functions per subdomain edge. The number of subdomains is fixed to $N=16$.]{
		\includegraphics[width=5.5cm]{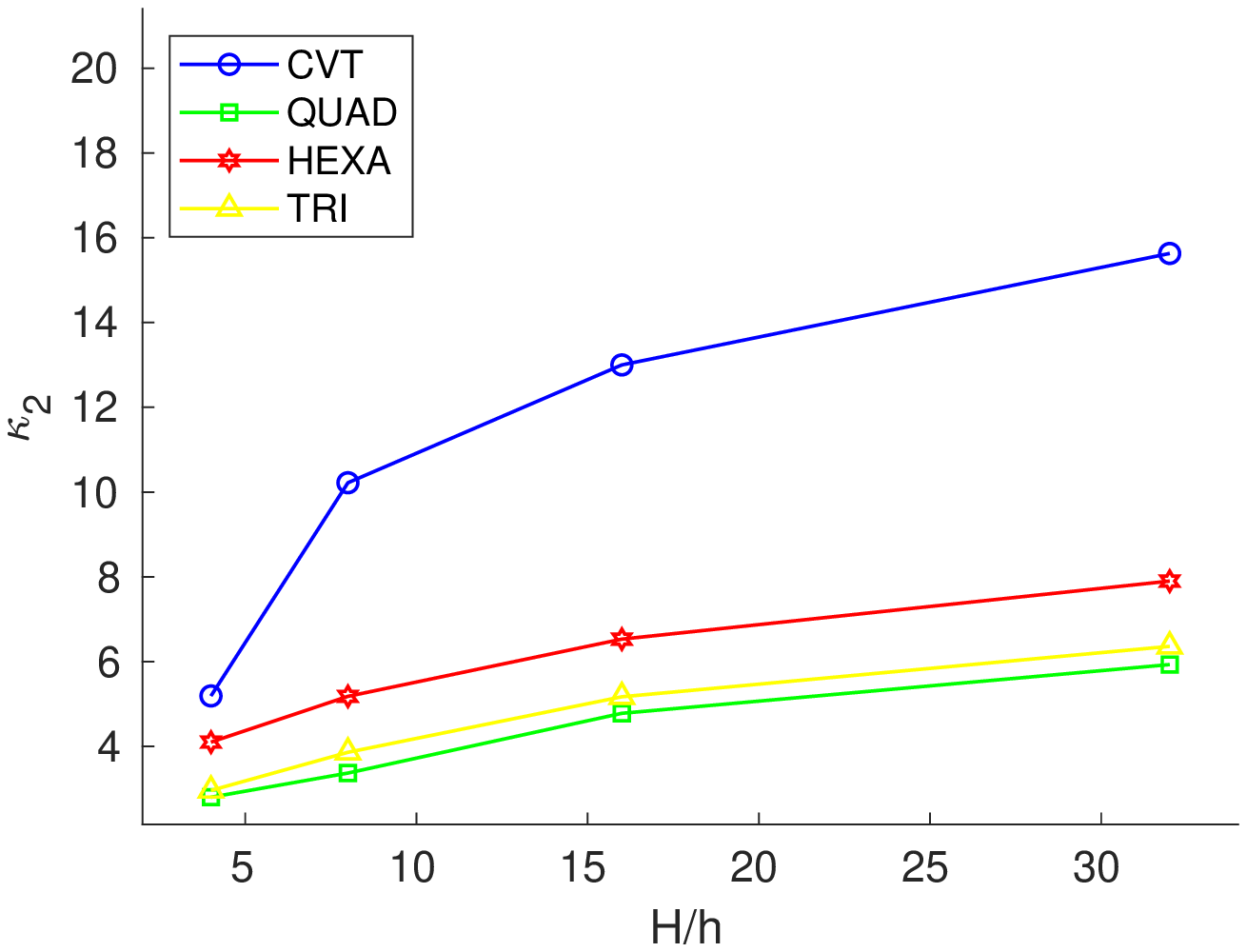}
	}
	\hspace{0.2cm} \subfigure[The primal space consists of the subdomain vertices and two basis functions per subdomain edge. The ratio $H/h$ is fixed to 4.]{
		\includegraphics[width=5.5cm]{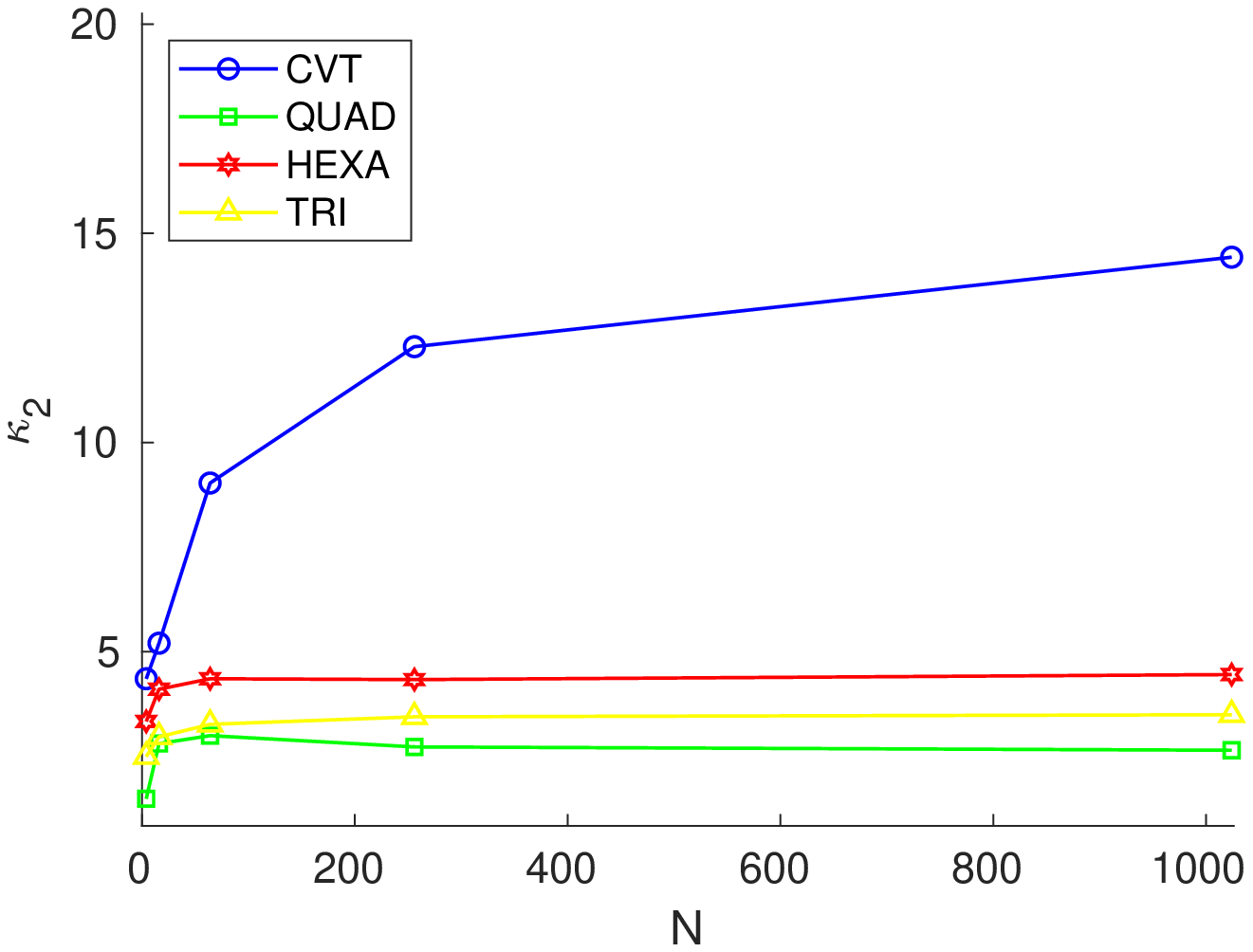}
	}
	\\
	\subfigure[The primal space consists of the subdomain vertices and one basis function per subdomain edge. The number of subdomains is fixed to $N=16$.]{
		\includegraphics[width=5.5cm]{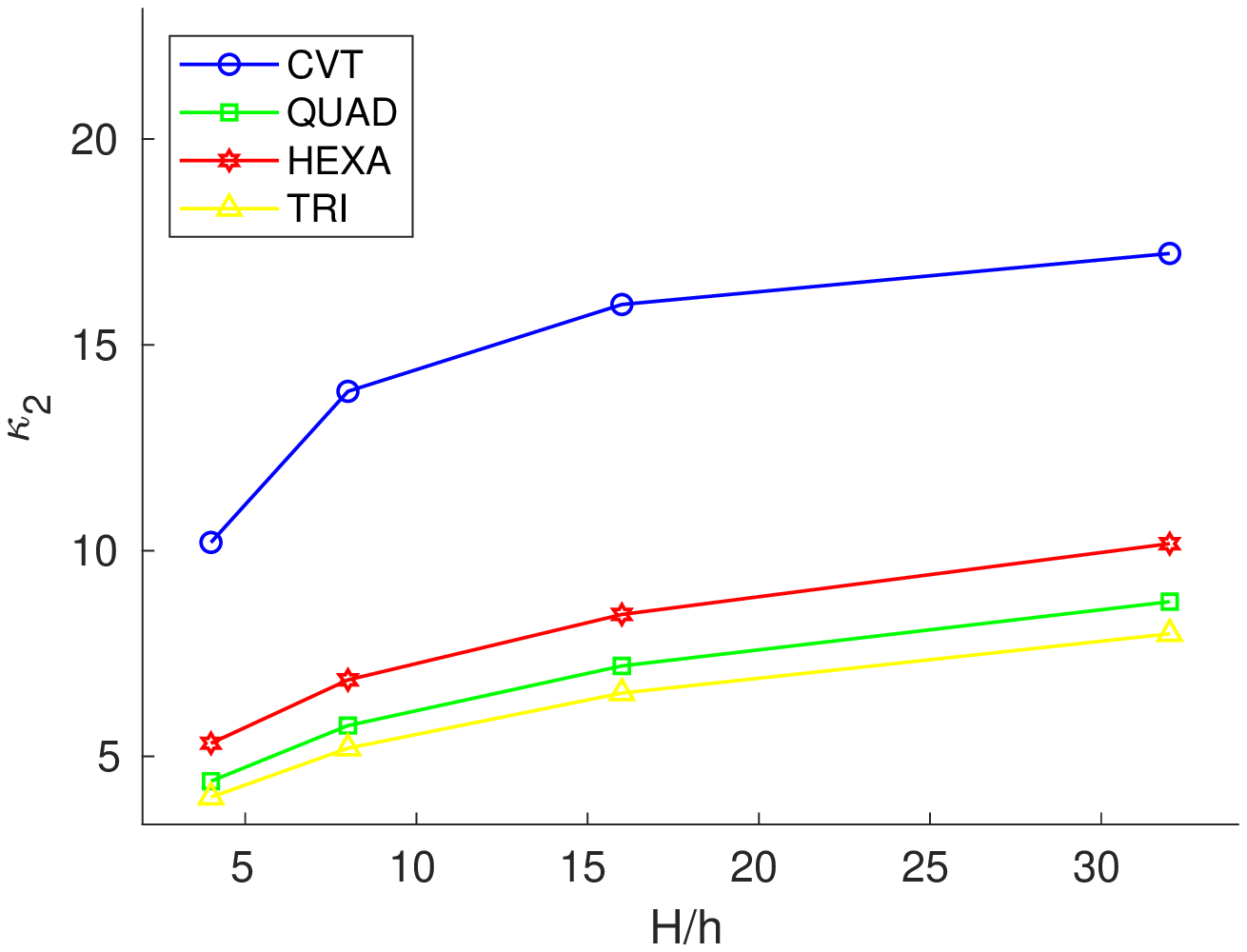}
	}
	\hspace{0.2cm} \subfigure[The primal space consists of the subdomain vertices and one basis function per subdomain edge. The ratio $H/h$ is fixed to 4.]{
		\includegraphics[width=5.5cm]{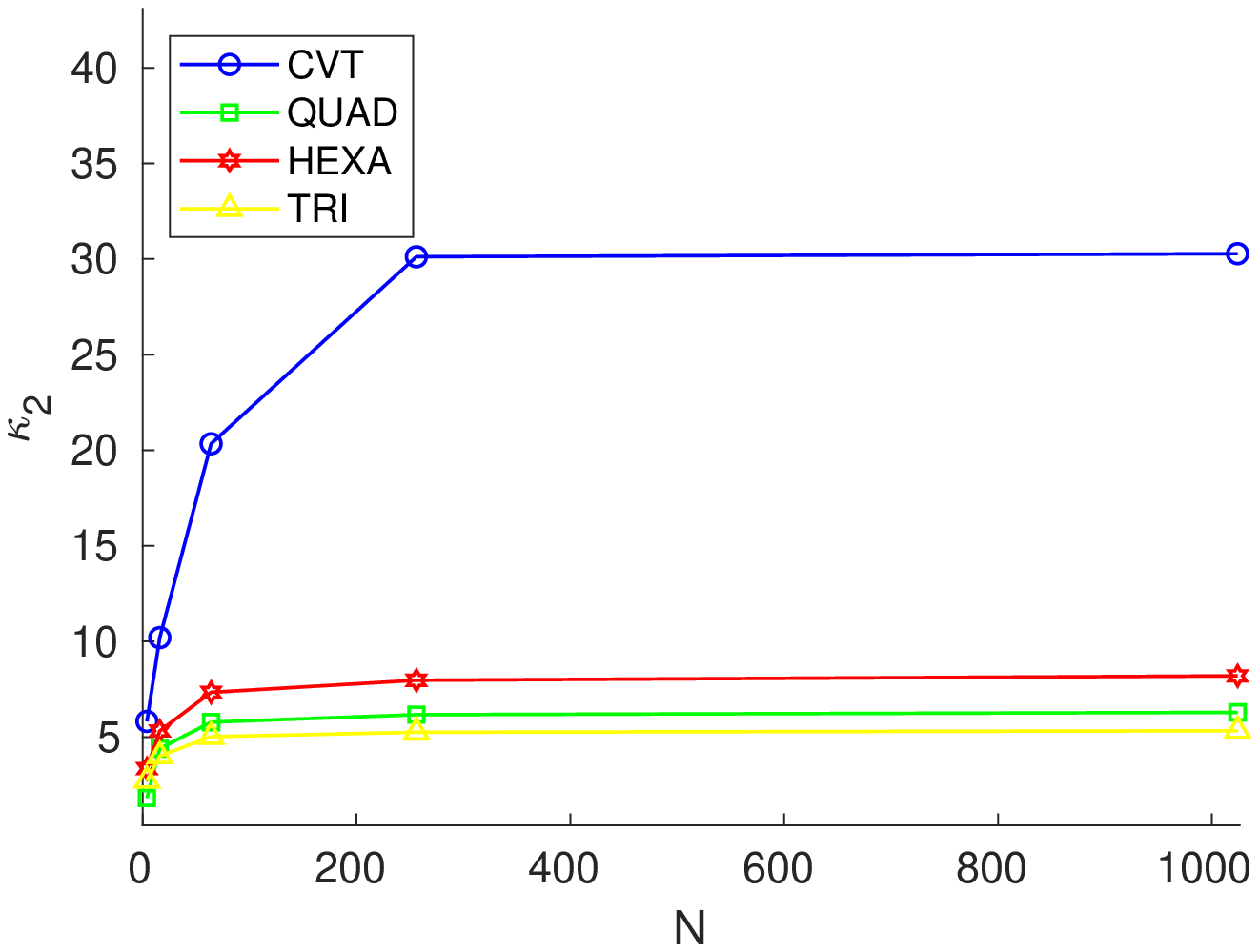}
	}
	\caption{Plots of the spectral condition number ($\kappa_2$) of the BDDC preconditioned linear system as a function of the ratio $H/h$ (left) and of the number of subdomains $N$ (right) for different types of mesh and choices of the primal space.}
	\label{k2fig}
\end{figure}

\begin{figure}[!t]	
	\centering
	\subfigure[QUAD meshes, 16 subdomains.]{
		\includegraphics[width=5.5cm]{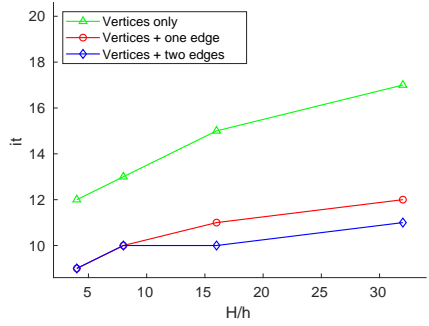}
	}
	\hspace{0.2cm} \subfigure[QUAD meshes, fixed local size (H/h=4).]{
		\includegraphics[width=5.5cm]{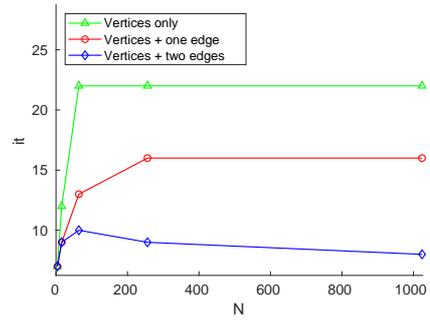}
	}\\
	\subfigure[HEXA meshes, 16 subdomains.]{
		\includegraphics[width=5.5cm]{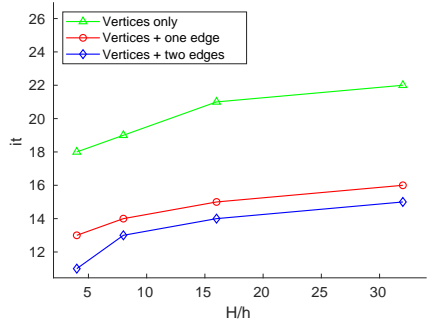}
	}
	\hspace{0.2cm} \subfigure[HEXA meshes, fixed local size (H/h=4).]{
		\includegraphics[width=5.5cm]{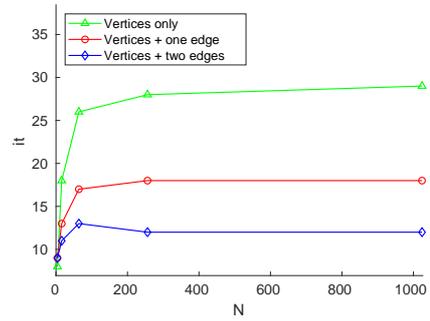}
	}
	\caption{Plots of the PCG iteration counts of the BDDC preconditioner for different choices of primal space on quadrilateral (QUAD) and hexagonal (HEXA) meshes.}
	\label{itfig}
\end{figure}

\begin{figure}[!t]
	\subfigure[TRI meshes, 16 subdomains.]{
		\includegraphics[width=5.5cm]{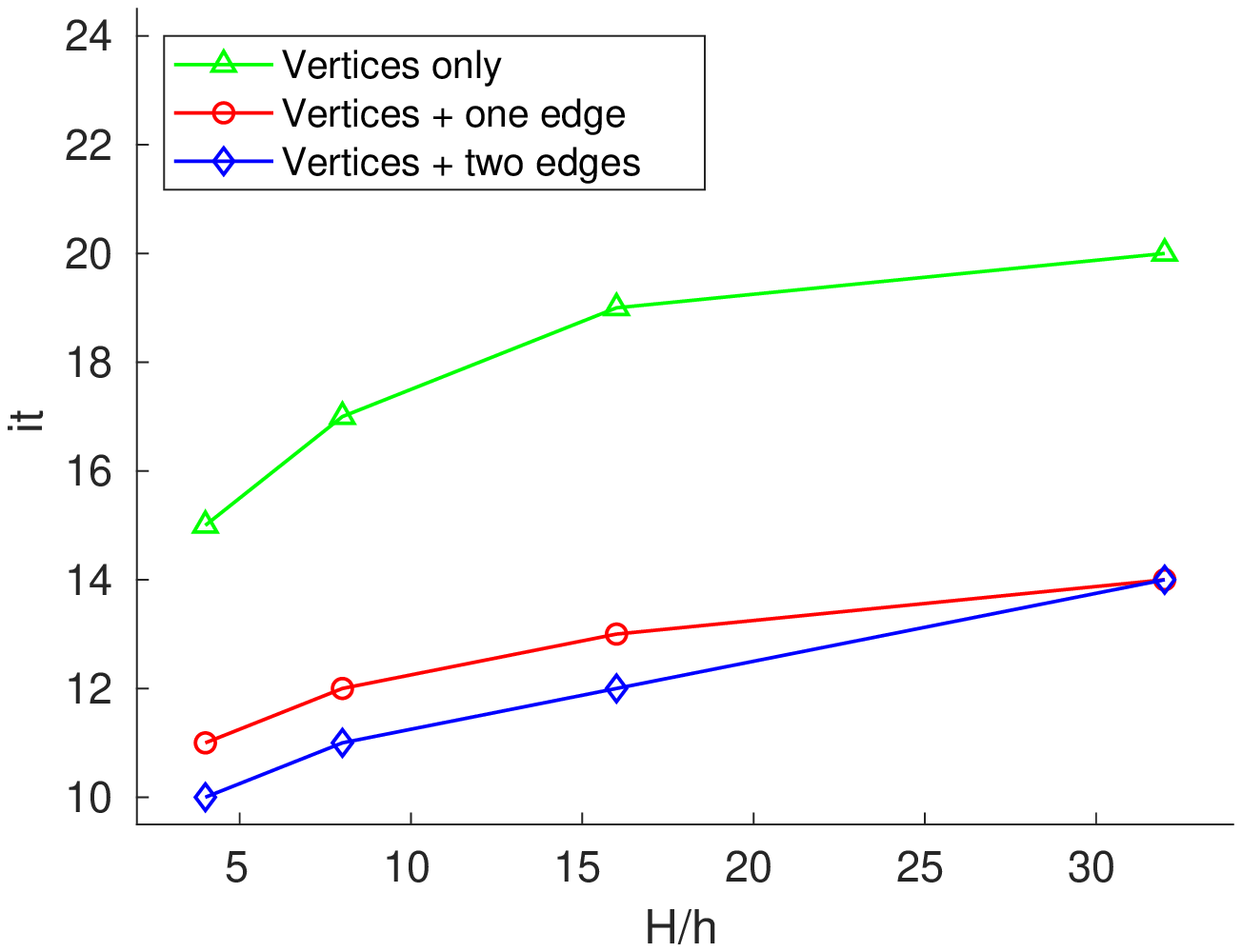}
	}
	\hspace{0.2cm} \subfigure[TRI meshes, fixed local size (H/h=4).]{
		\includegraphics[width=5.5cm]{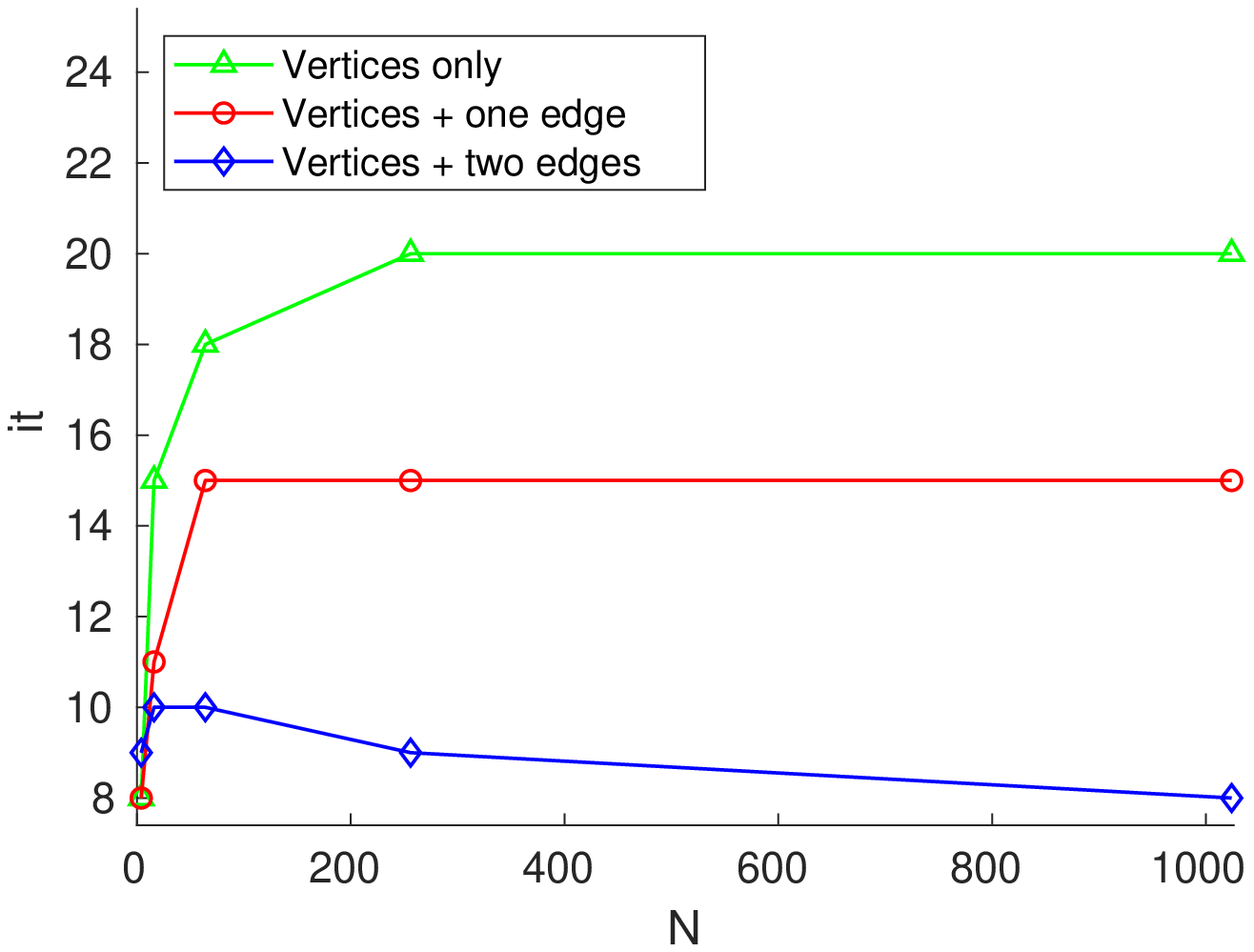}
	}
	\\
	\subfigure[CVT meshes, 16 subdomains.]{
		\includegraphics[width=5.5cm]{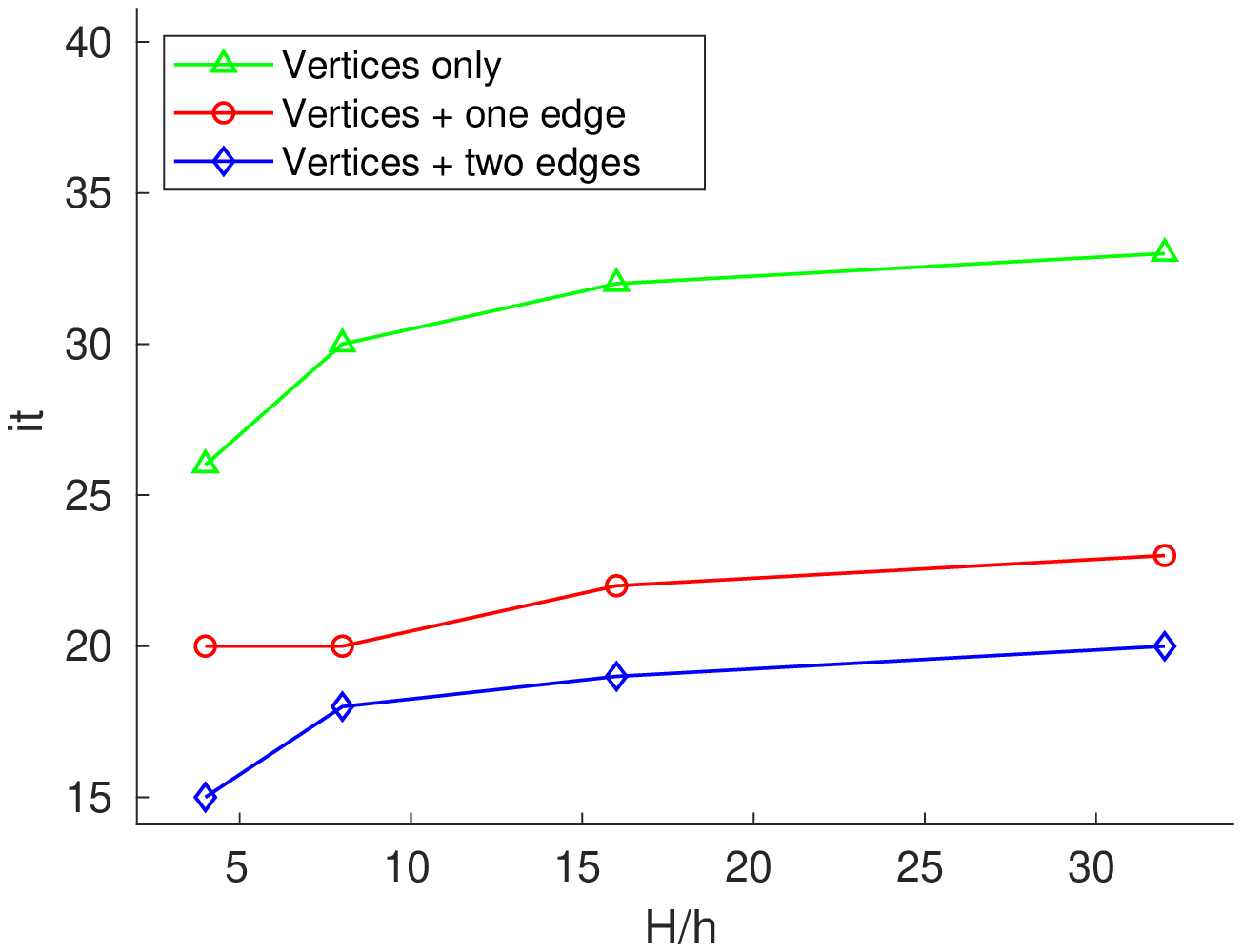}
	}
	\hspace{0.2cm} \subfigure[CVT meshes, fixed local size (H/h=4).]{
		\includegraphics[width=5.5cm]{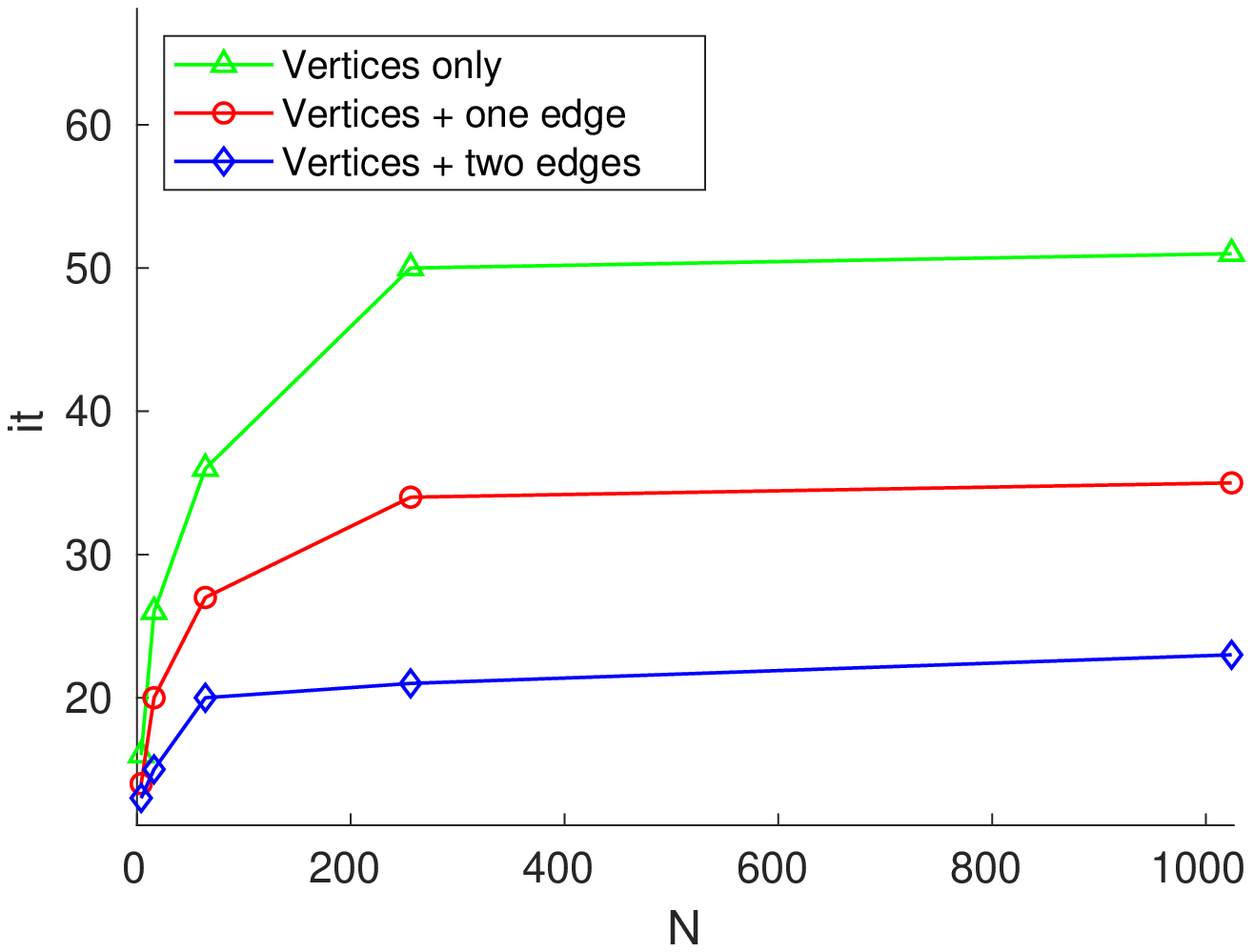}
	}
	\caption{Plots of the PCG iteration counts of the BDDC preconditioner for different choices of primal space on triangular (TRI) and Voronoi (CVT) meshes.}
	\label{itfig_2}
\end{figure}

In Figure \ref{k2fig}, we plot the spectral condition number of the BDDC preconditioner with the two different choices of primal constraints that satisfy the assumptions.
The left column displays an optimality test, fixing at 16 the number of subdomains and increasing the ratio $H/h$. We observe the logarithmic growth of the condition number.
The right column displays a weak scalability test, fixing the ratio $H/h=4$  and increasing the number of subdomains. In this case we see that the condition number remains bounded when the number of subdomains increases. We observe a worse behavior for the Voronoi meshes due to the fact that the boundary of the subdomains are quite irregular.

Finally, in Figures \ref{itfig} and \ref{itfig_2}, we plot the PCG iteration counts of the BDDC preconditioner for different choices of primal constraints and meshes.
The left column reports an optimality test with 16 subdomains and we observe that the logarithmic growth is respected, with a smaller number of iterations when the coarse space is enriched.
The right column displays the number of PCG iterations for a fixed local problem size ($H/h=4$) and we observe that the number of iterations remains bounded when the number of subdomains increase, again with a smaller number of iterations for richer primal spaces.

\section{Conclusions}
\label{sec:11}
In this work, we have analyzed BDDC preconditioners to solve the saddle-point linear system deriving from a divergence free VEM discretization of the steady two-dimensional Stokes equations. The numerical tests have validated the convergence estimates, showing the scalability and quasi-optimality of the algorithm, under appropriate choices of the primal coarse space. We have also obtained a better behavior and a faster convergence of the method for an enriched primal space, easy to implement.

\section{Acknowledgements}
The Authors are grateful to INdAM-GNCS for the support, we are also grateful to Giuseppe Vacca who provided us the initial code for the VEM discretization of the Stokes system.

\section*{Data Availability Statement}
The numerical experiments have been performed using an in-house matlab code available upon request to the authors.

%
\section*{Conflict of interest}

The authors declare that they have no conflict of interest.

\bibliographystyle{spmpsci}      
\bibliography{literature}   

\begin{thebibliography}{10}

\bibitem{antoniettiMasV.2018}
P.~F. Antonietti, L.~Mascotto, and M.~Verani.
\newblock A multigrid algorithm for the p-version of the virtual element
  method.
\newblock {\em {ESAIM}: Math. Model. Numer. Anal.}, 52(1):337--364, 2018.

\bibitem{beirao2013basic}
L.~Beir{\~a}o~da Veiga, F.~Brezzi, A.~Cangiani, G.~Manzini, L.~D. Marini, and
  A.~Russo.
\newblock Basic principles of virtual element methods.
\newblock {\em Math. Mod. Meth. Appl. Sci.}, 23(1):199--214, 2013.

\bibitem{da2017divergence}
L.~Beir{\~a}o~da Veiga, C.~Lovadina, and G.~Vacca.
\newblock Divergence free virtual elements for the stokes problem on polygonal
  meshes.
\newblock {\em ESAIM: Math. Mod. Numer. Anal.}, 51(2):509--535, 2017.

\bibitem{bertoluzza2017bddc}
S.~Bertoluzza, M.~Pennacchio, and D.~Prada.
\newblock Bddc and feti-dp for the virtual element method.
\newblock {\em Calcolo}, 54(4):1565--1593, 2017.

\bibitem{bertoluzza2020}
S.~Bertoluzza, M.~Pennacchio, and D.~Prada.
\newblock {FETI}-{DP} for the three dimensional virtual element method.
\newblock {\em {SIAM} J. Numer. Anal.}, 58(3):1556--1591, 2020.

\bibitem{boffi2013mixed}
D.~Boffi, F.~Brezzi, and M.~Fortin.
\newblock {\em Mixed finite element methods and applications}, volume~44.
\newblock Springer, 2013.

\bibitem{bramble1990domain}
J.~H. Bramble and J.~E. Pasciak.
\newblock A domain decomposition technique for stokes problems.
\newblock {\em Appl. Numer. Math.}, 6(4):251--261, 1990.

\bibitem{brenner2007}
S.~C. Brenner and L.-Y. Sung.
\newblock Bddc and feti-dp without matrices or vectors.
\newblock {\em Comput. Meth. Appl. Mech. Eng.}, 196(8):1429--1435, 2007.

\bibitem{calvo.2018}
J.~G. Calvo.
\newblock On the approximation of a virtual coarse space for domain
  decomposition methods in two dimensions.
\newblock {\em Math. Mod. Meth. Appl. Sci.}, 28(07):1267--1289, 2018.

\bibitem{calvo.2019}
J.~G. Calvo.
\newblock An overlapping schwarz method for virtual element discretizations in
  two dimensions.
\newblock {\em Comput. Math. Appl.}, 77(4):1163--1177, 2019.

\bibitem{canuto2014}
C.~Canuto, L.~F. Pavarino, and A.~B. Pieri.
\newblock {BDDC} preconditioners for continuous and discontinuous galerkin
  methods using spectral/hp elements with variable local polynomial degree.
\newblock {\em {IMA} J. Numer. Anal.}, 34(3):879--903, 2014.

\bibitem{dassiS.2020b}
F.~Dassi and S.~Scacchi.
\newblock Parallel block preconditioners for three-dimensional virtual element
  discretizations of saddle-point problems.
\newblock {\em Comput. Meth. Appl. Mech. Eng.}, 372:113424, 2020.

\bibitem{dassiZS2022}
F.~Dassi, S.~Zampini, and S.~Scacchi.
\newblock Robust and scalable adaptive {BDDC} preconditioners for virtual
  element discretizations of elliptic partial differential equations in mixed
  form.
\newblock {\em Comput. Meth. Appl. Mech. Eng.}, 391:114620, 2022.

\bibitem{dohrmann2003}
C.~R. Dohrmann.
\newblock A preconditioner for substructuring based on constrained energy
  minimization.
\newblock {\em SIAM J. Sci. Comput.}, 25(1):246--258, 2003.

\bibitem{dohrmann2016bddc}
C.~R. Dohrmann and O.~B. Widlund.
\newblock A {BDDC} algorithm with deluxe scaling for three-dimensional {H}
  (curl) problems.
\newblock {\em Comm. Pure Appl. Math.}, 69(4):745--770, 2016.

\bibitem{dryja2007}
M.~Dryja, J.~Galvis, and M.~Sarkis.
\newblock {BDDC} methods for discontinuous galerkin discretization of elliptic
  problems.
\newblock {\em J. Complex.}, 23(4-6):715--739, 2007.

\bibitem{hofer2018}
C.~Hofer.
\newblock Analysis of discontinuous galerkin dual-primal isogeometric tearing
  and interconnecting methods.
\newblock {\em Math. Mod. Meth. Appl. Sci.}, 28(1):131--158, 2017.

\bibitem{kim2009}
H.~H. Kim, M.~Dryja, and O.~B. Widlund.
\newblock A {BDDC} method for mortar discretizations using a transformation of
  basis.
\newblock {\em {SIAM} J. Numer. Anal.}, 47(1):136--157, 2009.

\bibitem{klawonn2006}
A.~Klawonn and O.~B. Widlund.
\newblock Dual-primal {FETI} methods for linear elasticity.
\newblock {\em Comm. Pure Appl. Math.}, 59(11):1523--1572, 2006.

\bibitem{litu2013}
J.~Li and X.~Tu.
\newblock A nonoverlapping domain decomposition method for incompressible
  stokes equations with continuous pressures.
\newblock {\em SIAM J. Numer. Anal.}, 51(2):1235--1253, 2013.

\bibitem{li2006bddc}
J.~Li and O.~Widlund.
\newblock Bddc algorithms for incompressible stokes equations.
\newblock {\em SIAM J. Numer. Anal.}, 44(6):2432--2455, 2006.

\bibitem{li2006feti}
J.~Li and O.~B. Widlund.
\newblock Feti-dp, bddc, and block cholesky methods.
\newblock {\em Int. J. Numer. Meth. Eng.}, 66(2):250--271, 2006.

\bibitem{oh2017}
D.-S. Oh, O.~B. Widlund, S.~Zampini, and C.~R. Dohrmann.
\newblock {BDDC} algorithms with deluxe scaling and adaptive selection of
  primal constraints for {R}aviart--{T}homas vector fields.
\newblock {\em Math. Comput.}, 87(310):659--692, 2017.

\bibitem{smith2004domain}
A.~Smith, P.~Bj{\o}rstad, and W.~Gropp.
\newblock {\em Domain Decomposition: Parallel Multilevel Methods for Elliptic
  Partial Differential Equations}.
\newblock Cambridge University Press, 2004.

\bibitem{toselli2006domain}
A.~Toselli and O.~B. Widlund.
\newblock {\em Domain decomposition methods-algorithms and theory}, volume~34.
\newblock Springer Science \& Business Media, 2006.

\bibitem{tu2018}
X.~Tu and B.~Wang.
\newblock A {BDDC} algorithm for the stokes problem with weak galerkin
  discretizations.
\newblock {\em Comput. Math. Appl.}, 76(2):377--392, 2018.

\bibitem{widlund2021}
O.~B. Widlund, S.~Zampini, S.~Scacchi, and L.~F. Pavarino.
\newblock Block {FETI}{\textendash}{DP}/{BDDC} preconditioners for mixed
  isogeometric discretizations of three-dimensional almost incompressible
  elasticity.
\newblock {\em Math. Comp.}, 90(330):1773--1797, 2021.

\bibitem{zampini2014dual}
S.~Zampini.
\newblock Dual-primal methods for the cardiac bidomain model.
\newblock {\em Mathematical Models and Methods in Applied Sciences},
  24(04):667--696, 2014.

\end{thebibliography}

%
%

\end{document}